\documentclass[a4paper, 10pt]{article}

\usepackage{graphics,graphicx}
\usepackage{amssymb,amsmath}
\usepackage{fullpage}
\usepackage{enumitem}
\usepackage{xcolor}

\usepackage[capitalize]{cleveref}

\newtheorem{theorem}{Theorem}

\newtheorem{problem}{Problem}
\newtheorem{claim}{}[theorem]
\newtheorem{claim1}{}[problem]

\newtheorem{cor}{Corollary}

\newtheorem{rlt}{Theorem}

\usepackage{thmtools}
\newlist{lemmalist}{enumerate}{1}
\setlist[lemmalist]{label=(\roman{lemmalisti}),
	ref=\thelemma:$(\roman{lemmalisti})$,	noitemsep}
\declaretheorem[name=Lemma]{lemma}

\Crefname{lemmalist}{Lemma}{Lemmas}
\addtotheorempostheadhook[lemma]{\crefalias{lemmalisti}{lemmalist}}

\def \no {\noindent}
 
 \def \sm {\setminus}
 \def \es {\emptyset}

\newenvironment{proof}[1][]%
{\noindent {\setcounter{equation}{0}\it Proof.
}{#1}{}}{\hfill$\Box$\vspace{2ex}}

\def\longbox#1{\parbox{0.85\textwidth}{#1}}

\begin{document}
\title{On  graphs with no induced $P_5$ or $K_5-e$}

\author{ Arnab Char\thanks{Computer Science Unit, Indian Statistical
Institute, Chennai Centre, Chennai 600029, India. } \and T.~Karthick\thanks{Corresponding author, Computer Science Unit, Indian Statistical
Institute, Chennai Centre, Chennai 600029, India. Email: karthick@isichennai.res.in. ORCID: 0000-0002-5950-9093. }}

\date{\today}

\maketitle

\begin{abstract}

In this paper\footnote{This paper is dedicated to the memory of Professor Frederic Maffray on his death anniversary.}, we are interested in some problems related to  chromatic number and  clique number  for the class of $(P_5,K_5-e)$-free graphs, and   prove the following.

\begin{enumerate}[label = {($\alph*$)}]
\item If $G$ is a connected ($P_5,K_5-e$)-free graph with $\omega(G)\geq 7$, then either $G$ is  the complement of a bipartite graph or $G$ has a clique cut-set. Moreover, there is a connected  ($P_5,K_5-e$)-free   imperfect graph $H$ with $\omega(H)=6$ and has no clique cut-set. This strengthens a result of Malyshev and Lobanova [Disc. Appl. Math. 219 (2017) 158--166].
     \item If $G$ is a  ($P_5,K_5-e$)-free graph with $\omega(G)\geq 4$, then $\chi(G)\leq \max\{7, \omega(G)\}$. Moreover, the bound is tight when $\omega(G)\notin \{4,5,6\}$. This result together with  known results partially answers a question of Ju and Huang [arXiv:2303.18003 [math.CO] 2023], and also improves a result of Xu [Manuscript  2022].

    \end{enumerate}
While  the \textsc{Chromatic Number Problem} is known to be $NP$-hard for the class of $P_5$-free graphs, our results together with some known results imply that the \textsc{Chromatic Number Problem} can be solved in polynomial time for the class of ($P_5,K_5-e$)-free graphs which may be independent interest.
\end{abstract}

\no{\bf Keywords}: Vertex coloring; Chromatic number; Clique number; $P_5$-free graphs.

\section{Introduction}

All graphs in this paper are finite, simple and undirected.
 For a positive integer $t$, let $K_t$ and $P_t$ respectively denote
the complete graph and the  chordless  path on $t$ vertices. For an integer
$t \geq 3$, $C_t$  denotes the chordless cycle on $t$ vertices. A $K_t-e$ is the graph obtained from $K_t$ by removing an edge. For
two vertex disjoint graphs $G_1$ and $G_2$, $G_1 + G_2$ is the disjoint union of $G_1$ and $G_2$. Given a graph $G$, let $tG$ denote the (disjoint) union of $t$ copies of $G$; for instance, $3K_1$ denotes the graph that consists of three disjoint copies of $K_1$.   We say that a graph $G$ {\it contains} a graph $H$ if $G$ has an induced subgraph which is isomorphic to $H$. Given a graph $H$, a graph is {\it $H$-free} if it does not contain $H$. A graph is
{\it $(H_1, H_2, \ldots, H_{\ell})$-free} if it does not contain  $H_i$ for each $i$.  A {\it clique} in a graph $G$  is a set of mutually adjacent vertices in $G$. A \emph{clique cut-set} in a graph $G$  is
a clique   in $G$ whose removal increases the number of components
of $G$. For any two disjoint sets of vertices $X$ and $Y$ of a graph $G$, we say that {\it $X$ is complete to $Y$} if every vertex in $X$
is adjacent to every vertex in $Y$; and we say that {\it $X$ is anticomplete to $Y$} if every vertex in $X$ is
nonadjacent to every vertex in $Y.$

We say that a graph $G$ admits a {\it $k$-coloring} if there is a function $f:V(G)\rightarrow \{1,2,\ldots, k\}$ such that for any  $uv \in E(G)$, we have $f(u)\neq f(v)$. The smallest $k$ such that $G$ admits a $k$-coloring is called the \emph{chromatic number} of $G$.
Given a graph $G$,  $\chi(G)$ denotes the chromatic number of $G$, and $\omega(G)$ denotes the clique
number of $G$ (which is the size of a largest clique in $G$). A graph $G$ is {\it $k$-colorable} if $\chi(G)\leq k$.
Clearly for any induced subgraph $H$ of $G$, we have $\chi(H)\geq \omega(H)$. A graph $G$ is {\it perfect} if   every induced subgraph $H$ of $G$ satisfies  $\chi(H) = \omega(H)$; otherwise $G$ is called an \emph{imperfect} graph.   The class of perfect graphs is well studied and received a wide attention for the past six decades. Examples of perfect graphs include bipartite graphs, complements of bipartite graphs, chordal  graphs, split graphs, $P_3$-free graphs etc., A celebrated  result of Chudnovsky et al.~\cite{SPGT} gives a characterization for the class of perfect graphs, and is now known as the Strong Perfect Graph Theorem.
  Gy\'arf\'as \cite{Gy87} extended the study of perfect graphs, and  introduced the class of $\chi$-bounded graphs.   An induced hereditary class  of graphs $\cal{G}$ is \emph{$\chi$-bounded} \cite{Gy87} if there is a function $f:\mathbb{N} \rightarrow \mathbb{N}$ with $f(1)=1$  and $f(x)\geq x$, for all $x\in \mathbb{N}$ such that every $G\in \cal{G}$ satisfies $\chi(G) \leq f(\omega(G))$; and if such a function $f$ exists, then $f$ is called a \emph{$\chi$-binding function} for $\cal{G}$. A well-known $\chi$-bounded class of graphs is the class of perfect graphs with $f(x)=x$ as  $\chi$-binding function.  In general, $\chi$-binding functions do not exist for an arbitrary class of graphs; see \cite{SS-Survey}.

  Gy\'arf\'as \cite{Gy87} conjectured that, for any tree $T$,  the class of $T$-free graphs is $\chi$-bounded, and showed that the conjecture holds for the class of $P_t$-free graphs, for any $t$. Esperet et al.  \cite{ELMM}  showed that every $P_5$-free graph $G$ with $\omega(G) \geq 3$ satisfies $\chi(G) \leq 5 \cdot 3^{\omega(G)-3}$, and the
bound is tight when $\omega(G) = 3$. While the class of $P_4$-free graphs is long been known to be perfect, it is unknown   whether the class of $P_5$-free graphs admits a polynomial $\chi$-binding function or not \cite{TP-P5}.    Existence of a polynomial $\chi$-binding function for the class of $P_5$-free graphs implies that the Erd\"os-Hajnal conjecture is true for the class of $P_5$-free graphs; see \cite{SS-Survey}.   The second author with Choudum and Shalu \cite{CKS} conjectured that there is a constant $c>0$ such that every $P_5$-free graph $G$ satisfies $\chi(G) \leq c~ \omega(G)^2$.  A recent result of Scott et al.~\cite{SS-P5}  gives the best  known bound for the class of $P_5$-free graphs which states  that such class of graphs admits a quasi-polynomial $\chi$-binding function. Indeed, they showed that every $P_5$-free graph $G$ with $\omega(G) \geq 3$ satisfies $\chi(G) \leq \omega(G)^{\log_2\omega(G)}$.  On the other hand, since the class of ($3K_1, 2K_2$)-free graphs does not admit a linear $\chi$-binding function \cite{BRSV}, and since the class of $3K_1$-free graphs  is a subclass of the class of $P_5$-free graphs, we conclude that the class of $P_5$-free graphs too does not admit a linear $\chi$-binding function.
Moreover, if $H$ is any  graph with independence number $\alpha(H) \geq 3$ or if $H$ contains an induced $2K_2$,  then the   class of ($P_5, H$)-free graphs  does not  admit a linear $\chi$-binding function.

Here  we  are interested in  $\chi$-bounded   ($P_5, H$)-free graphs,  where $H$ is any $2K_2$-free graph on five vertices with independence number $\alpha(H) \leq 2$. Clearly $H\in \{$$K_5$, $C_5$, $\overline{P_5}$/house,  $K_5-e$,  $K_4+K_1$, $4$-wheel, gem, paraglider,   HVN, flag, kite/co-chair$\}$; see Figure~\ref{fig}.  Below we give a short survey on the   best known chromatic bound for these  classes  of graphs. (Note that if $H\in \{$$K_5$,  $K_5-e$,  $K_4+K_1$,  HVN, flag$\}$, then the class of  ($P_5, H$)-free graphs generalizes  the class of ($P_5$, $K_4$)-free graphs \cite{ELMM}.)

\begin{figure}[t]
\centering
 \includegraphics{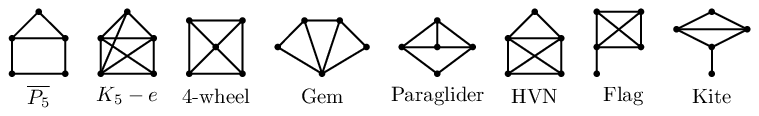}
\caption{Some $5$-vertex graphs with independence number 2.}\label{fig}
\end{figure}

\begin{itemize}[leftmargin=0.4cm]\itemsep=0pt
\item  Esperet et al.~\cite{ELMM} showed that every ($P_5,K_3$)-free graph $G$ satisfies $\chi(G)\leq 3$, and that every ($P_5, K_4$)-free graph $G$ satisfies $\chi(G)\leq 5$. Moreover these bounds are tight.  Further  they showed that every ($P_5, K_5)$-free graph $G$ satisfies $\chi(G) \leq 15$.   The problem of finding a tight $\chi$-bound for the class of ($P_5, K_5$)-free graphs  is open.

\item  Chudnovsky and Sivaraman~\cite{Chud-Siva} showed that every ($P_5, C_5$)-free graph $G$ satisfies $\chi(G)\leq 2^{\omega(G)}$, and the problem of finding a polynomial   $\chi$-binding function   for the class of ($P_5, C_5$)-free graphs  is open.

    \item  Fouquet et al.~\cite{Fouquet} showed that every ($P_5, \overline{P_5}$)-free graph $G$ satisfies $\chi(G)\leq \binom{\omega(G)+1}{2}$. They also showed that there is no linear $\chi$-binding function  for the class of ($P_5, \overline{P_5}$)-free graphs. Indeed, they constructed  an infinite family of ($P_5, \overline{P_5}$)-free graphs $G$ with $\chi(G) \geq \omega(G)^ {\log_25-1}$.

        \item If $G$ is a  $(P_5, K_4+K_1)$-free graph, then since the set of nonneighbors of any vertex in $G$ induces a $(P_5, K_4)$-free graph, an easy induction hypothesis on $\omega(G)$ together with the first item shows  that $\chi(G)\leq 5\omega(G)$. However, the problem of finding a  tight $\chi$-bound   for the class of ($P_5, K_4+K_1$)-free graphs  is open. It is known that \cite{Joos} every $(3K_1, K_4+K_1)$-free graph $G$ satisfies $\chi(G)\leq \frac{7}{4}\omega(G)$.

            \item In \cite{AK1}, the authors of the current paper showed that every ($P_5$, $4$-wheel)-free graph $G$ satisfies $\chi(G)\leq \frac{3}{2}\omega(G)$, and   that there are infinitely many  ($P_5, W_4$)-free graphs $H$ with $\chi(H)\geq \frac{10}{7}\omega(H)$.

    \item Chudnovsky et al.~\cite{CKMM} showed that every ($P_5$, gem)-free graph $G$ satisfies $\chi(G)\leq \lceil\frac{5\omega(G)}{4}\rceil$, and that the bound is tight.

\item The second author with Huang \cite{SK}  showed that every  ($P_5$, paraglider)-free graph $G$ satisfies $\chi(G)\leq \lceil\frac{3\omega(G)}{2}\rceil$, and that the bound is attained by the complement of
the   5-regular Clebsch graph on 16 vertices.  We also gave a complete characterization of a  ($P_5$, paraglider)-free graph $G$  that satisfies $\chi(G)> \frac{3}{2}\omega(G)$, and constructed an infinite family ${\cal L}$ of ($P_5$, paraglider)-free graphs such that every graph $G \in {\cal L}$ satisfies $\chi(G)= \lceil\frac{3\omega(G)}{2}\rceil -1$.

\item If $G$ is a ($P_5$, HVN)-free graph with $\omega(G)\leq 3$, then it follows from the first item that $\chi(G)\leq 5$.   Gei{\ss}er \cite{Geiber} showed that  every ($P_5$, HVN)-free graph $G$ with $\omega(G)\geq 4$ satisfies $\chi(G)\leq \omega(G)+1$. Moreover, these bounds are tight.

\item  If $G$ is a ($P_5$, flag)-free graph with $\omega(G)\leq 3$, then it follows from the first item that $\chi(G)\leq 5$.   In \cite{AK2}, we showed that every  ($P_5$,~flag)-free graph $G$ with $\omega(G)\geq 4$ satisfies $\chi(G)\leq \max\{8, 2\omega(G)-3\}$,  and the bound is tight for $\omega(G)\in \{4,5,6\}$.  We also constructed examples of such graphs $G$ with $\omega(G)=k$ and $\chi(G)= \lfloor \frac{3k}{2}\rfloor$.  Moreover, we conjectured that every ($P_5$, flag)-free satisfies $\chi(G)\leq \max\{8,\lfloor \frac{3\omega(G)}{2}\rfloor\}$, and that the bound is tight.

\item     The second author with Huang and Ju \cite{HJK}  showed that every ($P_5$,~kite)-free graph $G$ with $\omega(G)\leq 6$ satisfies $\chi(G)\leq\lfloor\frac{3\omega(G)}{2}\rfloor$, and that the bound is tight. Further, we showed that every  ($P_5$,~kite)-free graph $G$ with $\omega(G)\geq 6$ satisfies $\chi(G)\leq 2\omega(G)-3$.  We also constructed examples of such graphs $G$ with $\omega(G)=k$ and $\chi(G)= \lfloor \frac{3k}{2}\rfloor$. Moreover, we conjectured that every ($P_5$, kite)-free graph $G$ satisfies $\chi(G)\leq \lfloor \frac{3\omega(G)}{2}\rfloor$, and that the bound is tight.

    \item From results of  Malyshev and Lobanova \cite{ML17}, and Kierstead \cite{Kier84} (see also \cite{KierSchm}), it follows that if $G$ is a connected $(P_5,K_5-e)$-free graph with $\omega(G) > 3 \cdot6^7$ and has no clique
cut-set, then   $\chi(G)\leq \omega(G)+1$. Recently, Xu \cite{Xu} claimed that every ($P_5, K_5-e$)-free graph $G$  satisfies $\chi(G)\leq \max\{13, \omega(G)+1\}$, and the bound is tight when $\omega(G)\geq 12$. However, the proof of the same seems to have some error as it is based on the result which states that  if a graph $G$ is  ($P_5, C_5, K_5-e$)-free and is not a complete graph, then $G$ is 10-colorable (which is obviously not true). For instance, the graph obtained from $K_t$, $t\geq 11$, by attaching a pendent vertex satisfies all   assumptions of the said result, but is not $10$-colorable. Moreover, the tight examples given by Xu \cite{Xu} for $\omega\geq 12$  is clearly not ($K_5-e$)-free.


   \end{itemize}
A class of graphs $\cal{G}$  is said to be {\it near optimal colorable} \cite{Ju-Huang} if there is a constant positive integer $c$ such that
every graph $G\in \cal{G}$  satisfies $\chi(G)\leq \max\{c, \omega(G)\}$. For any two graphs $H_1$ and $H_2$, Ju and Huang \cite{Ju-Huang} gave a characterization for the
near optimal colorability of $(H_1, H_2)$-free graphs with three exceptional cases, one of which is stated below:
\begin{problem} [\cite{Ju-Huang}] \label{prob}
  Decide whether the class of $(F, K_t-e)$-free graphs is near optimal colorable when $F$ is a forest  and $t\geq 4$.
\end{problem}

\cref{prob} seems to be difficult in general even when $F=P_{\ell}$, $\ell \geq 5$.  It is known that every ($P_5, K_4-e$)-graph $G$ satisfies $\chi(G)\leq \max\{3,\omega(G)\}$  \cite{Geiber}, and that every ($P_6, K_4-e$)-graph $G$ satisfies $\chi(G)\leq \max\{6,\omega(G)\}$  \cite{GHJM}.
However, it is unknown that whether the class of ($P_{t}, K_{t}-e$)-free graphs (where $t \geq 5$) is near  optimal colorable or not.

\medskip
 In this paper, we focus on the class of ($P_5,K_5-e$)-free graphs, and we prove the following theorem (and its proof is given in Section~4).

\begin{theorem}\label{main-thm}
	 Let $G$ be a connected ($P_5,K_5-e$)-free graph. Then the following hold:
 \begin{enumerate}\itemsep=0pt
 \item[(a)] If  $\omega(G)\geq 5$, then either $G$ is   the complement of a bipartite graph or $G$ has a clique cut-set or   $\chi(G)\leq 6$. \item[(b)] If $\omega(G)=4$, then either $G$ is   the complement of a bipartite graph or $G$ has a clique cut-set or   $\chi(G)\leq 7$.	\end{enumerate}
	\end{theorem}
As a corollary of this result, we strengthen a result of Malyshev and Lobanova \cite{ML17} which states that if $G$ is a connected $(P_5,K_5-e)$-free graph with $\omega(G) \leq 3\cdot 6^7 = 839808$, then either $G$ is $3K_1$-free or $G$ has a clique cut-set.    (Note that    the graph $C_5$ is an imperfect ($3K_1, K_5-e$)-free graph which  has no clique cut-set, and we refer to  Figure~\ref{extexp} for nontrivial examples.)
\begin{cor}\label{ML-imp}
If $G$ is a connected ($P_5,K_5-e$)-free graph with $\omega(G)\geq 7$, then either $G$ is  the complement of a bipartite graph or $G$ has a clique cut-set.	 Moreover, the assumption on the lower bound of $\omega$ is tight. That is, there is a connected  ($P_5,K_5-e$)-free   imperfect graph $H$ with $\omega(H)=6$ and has no clique cut-set.
\end{cor}
\begin{proof}
The first assertion is an immediate consequence of \cref{main-thm}. For the second assertion, consider the graph $H^*$ given in Figure~\ref{extexp}. Then $H^*$ is a connected  ($P_5,K_5-e$)-free   imperfect graph  with $\omega(H^*)=6$ and has no clique cut-set.
\end{proof}

Further, we have the following theorem which improves an earlier stated result of Xu \cite{Xu}.

\begin{figure}[t]
\centering
 \includegraphics[height=2.5cm]{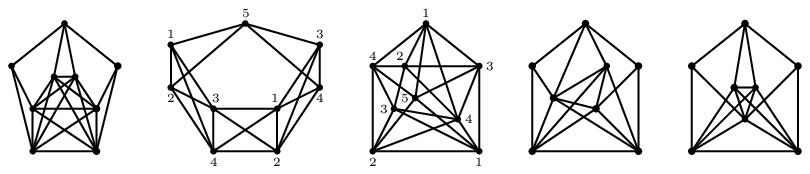}
\caption{The graphs $H^*$, $G_1$, $G_2$, $G_3$ and $G_4$ (left to right). In $G_1$ and $G_2$, the labels represent a $5$-coloring.}\label{extexp}
\end{figure}

\begin{theorem} \label{chr-bnd}
If $G$ is a  ($P_5,K_5-e$)-free graph with $\omega(G)\geq 4$, then $\chi(G)\leq \max\{7, \omega(G)\}$. Moreover, the bound is tight when $\omega(G)\notin \{4,5,6\}$.
\end{theorem}

\begin{proof}
Let $G$ be a  ($P_5,K_5-e$)-free graph with $\omega(G)\geq 4$. We prove the first assertion by induction on $|V(G)|$. We may assume that $G$ is  connected and is imperfect. Then from \cref{main-thm}, either $G$ has a clique cut-set or   $\chi(G)\leq 7$. If $\chi(G)\leq 7$, then we are done. So we may assume that $G$ has a clique cut-set, say $K$. Let $A$ and $B$ be a partition of $V(G) \sm K$ such that $A, B\neq \es$ and  $A$ is anticomplete to $B$. Then $\chi(G) =\max\{\chi(G[K\cup A]), \chi(G[K\cup B])\}$, and hence by induction hypothesis, we have $\chi(G) \leq \max\{\max\{7, \omega(G[K\cup A])\}, \max\{7, \omega(G[K\cup B])\}\} \leq  \max\{7, \omega(G)\}$.  This proves the first assertion.
To prove the second assertion, consider the graph  $G$
  that consists of a complete graph $K_t$ where $t\geq 7$, say $Q$ such that (a) each component of $G[V(G)\sm V(Q)]$ is
a complete graph, (b) for each component $K$ in $G[V(G)\sm V(Q)]$, there is a unique $v\in V(Q)$ such that $\{v\}$ is complete to $V(K)$, and (c) no other
edges in $G$.   Clearly  $G$ is a ($P_5,K_5-e$)-free perfect graph,  and so $\chi(G) = \omega(G) = t$. This proves \cref{chr-bnd}. \end{proof}

We remark that there are ($P_5,K_5-e$)-free graphs with $\omega=4$ and $\chi=5$.
For instance, consider the graphs $G_1$ and $G_2$  given in Figure~\ref{extexp}.
 Then clearly for $j\in \{1,2\}$:   each  $G_j$ is ($P_5$, $K_5-e$)-free with $\chi(G_j)\leq 5$ (see  Figure~\ref{extexp} for a $5$-coloring), and it is easy check that $\omega(G_j)=4$, $\alpha(G_j)=2$, and hence $\chi(G_j)\geq \lceil \frac{|V(G_j)|}{\alpha(G_j)} \rceil = 5$.

\smallskip
Next we have the following corollary that partially answers \cref{prob}. That is, every ($P_5,K_5-e$)-free graph is near  optimal colorable.

\begin{cor} \label{chr-bnd-final}
If $G$ is a  ($P_5,K_5-e$)-free graph, then $\chi(G)\leq \max\{7, \omega(G)\}$.
\end{cor}
\begin{proof}
Let $G$ be a  ($P_5,K_5-e$)-free graph. If $\omega(G) \leq  3$, then $G$ is ($P_5, K_4$)-free, and hence $\chi(G)\leq 5$ \cite{ELMM}. So we may assume that $\omega(G)\geq 4$. Now  the corollary follows from \cref{chr-bnd}.
\end{proof}

  \no{\bf Algorithmic aspects}. Given  a graph $G$  and a positive integer $k$, the $k$-\textsc{Colorability Problem} asks whether or not $G$ admits a $k$-coloring.  Given a graph $G$, the \textsc{Chromatic Number} (or \textsc{Minimum Vertex Coloring}) \textsc{Problem} asks whether or not $\chi(G)\leq k$. The $k$-\textsc{Colorability Problem} is well-known to be $NP$-complete for any fixed $k\geq 3$. Hence  the \textsc{Chromatic Number Problem} is known to be $NP$-hard in general, and is known to be $NP$-hard even for the class of $P_5$-free graphs \cite{KKTW}. We refer to a survey of Golovach et al.~\cite{GJPS} for more details and results. Using Lov\'asz theta function,
  Ju and Huang \cite{Ju-Huang} observed that if $\cal G$ is a given   hereditary class of graphs such that
every   $G \in \cal G$ satisfies   $\chi(G)\leq \max\{c, \omega(G)\}$ for some constant $c$, and if the $k$-\textsc{Colorability Problem} for $\cal G$ is polynomial time solvable for every fixed positive integer $k \leq c-1$, then the \textsc{Chromatic Number Problem} for
$\cal G$ can be solved in  polynomial time.  Since the $k$-\textsc{Colorability Problem} for the class of $P_5$-free graphs can be solved in  polynomial time for every fixed positive integer $k \leq 6$ \cite{HKLSS}, from \cref{chr-bnd-final}, we conclude that the \textsc{Chromatic Number Problem}  for the class of ($P_5,K_5-e$)-free graphs    can be solved in  polynomial time. We remark that this conclusion may also be obtained from \cref{main-thm} by using  clique separator decomposition techniques (see \cite{ML17}) and  a result of Ho\`ang et al.~\cite{HKLSS}.

\begin{figure}
\centering
 \includegraphics{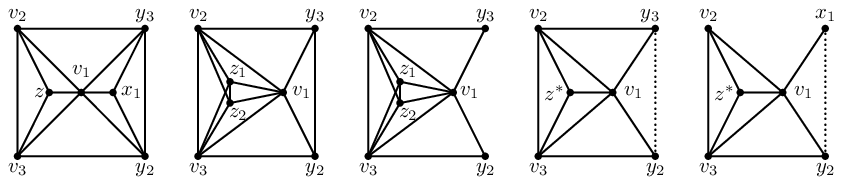}
\caption{ Graphs $F_1$, $F_2$, $F_3$,   $F_4$ and $F_5$   (left to right). In $F_4$ and $F_5$, the dotted lines between two vertices mean that  such  vertices may or may not be adjacent.}\label{fig-F123}
\end{figure}

In Section~\ref{not-ter}, we give some notation, terminology and  preliminaries which are used in this paper. The rest of the paper is devoted to the proof of \cref{main-thm}, and is given in Section~\ref{mainthm-proof}.  Indeed, the proof of \cref{main-thm} is based on some intermediate  results using certain special graphs; see Figure~\ref{fig-F123}. First we show that \cref{main-thm} holds if our graph $G$ contains one of   $F_1$, $F_2$ or $F_3$; see Section~\ref{GcontF123}. (In particular, in each case, we show that   either $G$ is  the complement of a bipartite graph or $G$ has a clique cut-set or  $\chi(G)\leq 5$).  Then we show that
  \cref{main-thm} holds when $G$ is ($F_1, F_2, F_3$)-free; see Section~\ref{mainthm-f123free}.

 Note that  the graph $G_2$ is an imperfect graph, has no clique cut-set and contains an $F_1$ with $\chi(G_2)=5$, the graph $G_3$ is an imperfect graph, has no clique cut-set and contains an $F_2$ with $\chi(G_3)=\omega(G_3)=5$, and the graph $G_4$ is an imperfect graph, has no clique cut-set and contains an $F_3$ with $\chi(G_4)=\omega(G_4)=5$.

\section{Notation, terminology and preliminaries}\label{not-ter}

 For missing notation and terminology, we refer to Bondy and Murty \cite{Bondy-Murty}.
Given a graph $G$, $V(G)$ denotes its vertex-set and $E(G)$ denotes its edge-set. A vertex $u$ is a \emph{neighbor} of $v$ in $G$ if $uv\in E(G)$.
Given a vertex $v\in V(G)$, $N_G(v)$ is the set of neighbors
of $v$ in $G$; we may drop the subscript $G$, and write $N(v)$ when the relevant graph is
unambiguous.  For a vertex subset $X\subseteq V(G)$, $G[X]$ denotes the subgraph of $G$ induced by $X$.    A vertex subset $X$ of $G$ is a \emph{homogeneous} set in $G$   if every vertex in $V(G)\sm X$ which has a neighbor in $X$ is complete to $X$.  A {\it stable set} in a graph $G$ is a set of mutually nonadjacent vertices in $G$,  and the {\it independence number}, denoted by $\alpha(G)$, is the size of a largest stable set in $G$.

 A {\it $5$-ring}~ is the graph
whose vertex-set can be partitioned into five non-empty stable sets $S_1, S_2, \ldots, S_5$ such that for each $i$ mod $5$, every vertex
in $S_i$ is adjacent to every vertex in $S_{i+1}\cup S_{i-1}$ and to no vertex in $S_{i+2}\cup S_{i-2}$.
  Given a graph $G$, a {\it $5$-ring-component} of   $G$ is a component of $G$ which is isomorphic to a $5$-ring, and a {\it big-component} of  $G$ is a component of $G$ which has at least two vertices.

 \smallskip
 For a positive integer $t$,   let $\mathbb{H}_t$ be the graph obtained from $K_{t+3}$ by adding a new vertex and joining it to exactly two vertices of $K_{t+3}$. Note that the graph $\mathbb{H}_1$ is isomorphic to HVN (see Figure~\ref{fig}).

 \smallskip

We will use  the following simple observations often. {\it Let $G$ be a $(K_5-e)$-free graph. Then:}
\begin{enumerate}[label={(O\arabic*)}]	
\item\label{uvadjnbd} {\it For any two adjacent vertices in $G$, say $u$ and $v$, $G[N(u)\cap N(v)]$ is $P_3$-free.}
		
{\it Proof}. If there is a $P_3$ in $G[N(u)\cap N(v)]$  with vertices,  say $p,q$ and $r$, then $\{p,q,r,u,v\}$ induces a $K_5-e$. So  \ref{uvadjnbd}  holds. $\Diamond$
		
\item\label{uvnonadjnbd} {\it  For any two nonadjacent vertices in $G$, say $u$ and $v$, $G[N(u)\cap N(v)]$ is $K_3$-free.}

{\it Proof}. If there is a $K_3$ in $G[N(u)\cap N(v)]$  with vertices,  say $p,q$ and $r$, then $\{p,q,r,u,v\}$ induces a $K_5-e$. So \ref{uvnonadjnbd} holds. $\Diamond$

\item\label{4setsclique} {\it If there are four mutually disjoint nonempty subsets of $V(G)$ which are complete to each other, then their union is a clique.}
			
			{\it Proof}. This follows from \ref{uvnonadjnbd}. $\Diamond$
	\end{enumerate}

We will also use the following known results.
	\begin{rlt}[\cite{ELMM}]\label{P5K3prop}
		(a) If $G$ is a ($P_5, K_3$)-free graph, then each component of $G$ is either bipartite or a $5$-ring.  In particular,  $\chi(G)\leq 3$. (b) If $G$ is a $(P_5,K_4)$-free graph, then  $\chi(G)\leq 5$.
	\end{rlt}

\begin{rlt}[\cite{Geiber}]\label{P5diamond-col}
	 Every $(P_5,K_4-e)$-free graph $G$ satisfies $\chi(G)\leq \max\{3, \omega(G)\}$.
	\end{rlt}

\section{($P_5,K_5-e$)-free graphs with $\omega \geq 4$}

To prove \cref{main-thm}, we begin by proving some simple properties when  a ($P_5,K_5-e$)-free graph contains a $K_3$, and use them   in the latter sections.  From now on, we assume that the arithmetic operations on the indices are in integer modulo $3$.

%
%

\subsection{Properties of $(P_5,K_5-e)$-free graphs that contain a $K_3$}\label{genprop}
Let $G$ be a connected ($P_5,K_5-e$)-free graph with vertex-set $V(G)$ and edge-set $E(G)$ that contains a $K_3$ induced by the vertices, say $v_1,v_2$ and $v_3$. Let $C:=\{v_1,v_2,v_3\}$. For $i\in \{1,2,3\}$, we define:
\begin{eqnarray*}
X_i &:=&\{u\in V(G)\sm C \mid N(u)\cap C=\{v_i\}\}, \\
Y_i&:=&\{u\in V(G)\sm C \mid N(u)\cap C= C\sm \{v_i\}\},\\
Z&:=& \{u\in V(G)\sm C \mid N(u)\cap C= C\}, \mbox{and}\\
L&:=& \{u\in V(G)\sm C \mid N(u)\cap C= \es\}.
\end{eqnarray*}
We let $X:=X_1\cup X_2\cup X_3$ and $Y:=Y_1\cup Y_2\cup Y_3$.  Then clearly $V(G)= C\cup X\cup Y \cup Z\cup L$. Moreover, the following properties hold, where $i\in \{1,2,3\}$, $i$ mod $3$.

\begin{claim1}\label{Zclq}
$C\cup Z$ is a clique.
\end{claim1}
{\it Proof of \cref{Zclq}}.~If there are nonadjacent vertices, say $z$ and $z'$ in $Z$, then $\{v_1, v_{2},v_{3},z, z'\}$ induces a $K_5-e$. So \ref{Zclq} holds. $\Diamond$
	
	\begin{claim1} \label{BZancom}  $G[Y_i]$ is $P_3$-free, and $Y$ is anticomplete to $Z$.  \end{claim1}	
	{\it Proof of \ref{BZancom}}.~Since $Y_i$ is complete to $\{v_{i+1}, v_{i-1}\}$,  $G[Y_i]$ is $P_3$-free (by \ref{uvadjnbd}).  Next, if there are adjacent vertices, say  (up to symmetry) $y\in Y_1$ and $z\in Z$, then $\{v_{2},v_{3},z, y,v_1\}$ induces a $K_5-e$; so $Y$ is anticomplete to $Z$. This proves \ref{BZancom}. $\Diamond$
	
	\begin{claim1} \label{AiBi+1homset}
	  The vertex-set of each component of $G[X_i\cup L]$ is a homogenous set in $G[X_i\cup X_{i+1}\cup Y_{i+2}\cup L]$. Likewise, the vertex-set of each component of $G[X_i\cup L]$ is a homogenous set in $G[X_i\cup X_{i+2}\cup Y_{i+1}\cup L]$. \end{claim1}	
	{\it Proof of  \ref{AiBi+1homset}}.~We prove the assertion for $i = 1$. Suppose to the contrary that there are vertices, say $p, q\in X_1\cup L$ and $r\in X_2 \cup Y_3$ such that
	$pq, pr \in E(G)$ and $qr \notin E(G)$. Then $\{q,p,r,v_2,v_3\}$ induces a $P_5$, a contradiction. So \ref{AiBi+1homset} holds. $\Diamond$
	
	\begin{claim1} \label{aZcard}  Each vertex in $X$ has at most one neighbor in $Z$. \end{claim1}	
	{\it Proof of  \ref{aZcard}}.~If there is a vertex in $X_i$, say $x$, which has two neighbors in $Z$, say $z$ and $z'$, then $zz'\in E(G)$ (by \ref{Zclq}), and then $\{v_i,z,z',v_{i+1},x\}$ induces a $K_5-e$. So \ref{aZcard} holds. $\Diamond$
	
	 \begin{claim1}\label{Atcorel} Suppose that there is a vertex $t\in L$ which has a neighbor in $X_i$. Then the following hold:
	 (a)  $\{t\}$ is complete to $X_{i+1}\cup X_{i+2}$. (b) $X_i$ is complete to $X_{i+1}\cup X_{i+2}$.
 \end{claim1}
{\it Proof of \ref{Atcorel}}.~We prove for $i = 1$. By our assumption, $t$ has a neighbor in $X_1$, say $x$.\\
  $(a)$:~If there is a vertex in $X_2$, say $p$,  such that $tp \notin E(G)$, then since $\{t, x, v_1, v_2, p\}$ does not induce a $P_5$, we have $xp \in E(G)$, and then $\{t, x, p, v_2, v_3\}$ induces a $P_5$; so $\{t\}$ is complete to $X_2$. Likewise, $\{t\}$ is complete to $X_3$. This proves \ref{Atcorel}:(a). $\diamond$

  \smallskip	
	\no{$(b)$}:~If there are nonadjacent vertices, say $p\in X_1$ and $q\in X_2 \cup X_3$, then $qt, pt\in E(G)$ (by \ref{Atcorel}:(a)), and then $\{p, t, q, v_2, v_3\}$ induces a $P_5$; so \ref{Atcorel}:(b) holds. $\Diamond$

 \begin{claim1}\label{BiBi+1bigcomancom}   If $Z\neq \es$, then the vertex-set of any big-component of $Y_i$ is anticomplete $Y_{i+1}\cup Y_{i-1}$. \end{claim1}
{\it Proof of  \ref{BiBi+1bigcomancom}}.~We will show  for $i=1$.  Let $z\in Z$. Suppose  that the assertion is not true. Then there are vertices, say $p,q\in Y_1$   and    $r\in Y_2\cup Y_3$ such that $pq, pr\in E(G)$. We may assume that $r\in Y_2$.  Now since $\{p,q,v_3,v_2,r\}$ does not induce a $K_5-e$, we have $qr\notin E(G)$, and then $\{q,p,r,v_1,z\}$ induces a $P_5$ (by \ref{BZancom}), a contradiction. So \ref{BiBi+1bigcomancom} holds. $\Diamond$

\begin{claim1} \label{lZcard} Each vertex in $L$ has at most two neighbors in $Z$.\end{claim1}
{\it Proof of \ref{lZcard}}.~If there is a vertex in $L$, say $t$, which has three neighbors in $Z$, say $z_1,z_2$ and $z_3$, then since $Z$ is a clique (by \ref{Zclq}),  $\{z_1,z_2,z_3,v_1,t\}$ induces a $K_5-e$. So \ref{lZcard} holds. $\Diamond$

\begin{claim1}\label{F2azancomLcorel} For $j\in \{i+1,i-1\}$, if there are vertices, say $p\in X_i\cup Y_{j}$, $q\in X_{j}\cup Y_{i}$, $z\in Z$  and $t\in L$ such that $pt\in E(G)$, then either $pq\in E(G)$ or $qt\in E(G)$. Further if $pz,qz,tz\notin E(G)$, then $pq,qt\in E(G)$.  \end{claim1}
		{\it Proof of \ref{F2azancomLcorel}}.~Otherwise,  one of $\{t,p,v_i,v_{j},q\}$ or $\{t,p,q,v_{j},z\}$  induces a $P_5$.  $\Diamond$

\begin{claim1}\label{K4XiYihomset}Let $S$ be the	vertex-set of a component of $G[Y_i]$, and let $R$ be the vertex-set of a component of $G[X_i]$. Suppose  that there is a vertex, say $z\in Z$, such that $\{z\}$ is anticomplete to $R$. Then $S$ is either complete to $R$ or anticomplete to $R$.  Moreover if $S$ is the vertex-set of a big-component of $Y_i$, and if $R$ is not anticomplete to $S$, then $G[R]$ is $P_3$-free.  \end{claim1}
		{\it Proof of \ref{K4XiYihomset}}.~If there are vertices, say $a,b\in R$ and $c\in S$  such that $ab,ac\in E(G)$ and $bc\notin E(G)$, then $\{b,a,c,v_{i+1},z\}$ induces a $P_5$ (by \ref{BZancom}); so  $R$ is a homogeneous set in $G[R\cup S]$. Also, if there are vertices, say  $p,q\in S$ and $r\in R$  such that $pq,pr\in E(G)$ and $qr\notin E(G)$, then $\{q,p,r,v_i,z\}$ induces a $P_5$ (by \ref{BZancom}); so $S$ is homogeneous set in $G[R\cup S]$. This implies that $S$ is either complete to $R$ or anticomplete to $R$. This  proves the first assertion. The second assertion follows from the first assertion and from \ref{uvadjnbd}. This proves \cref{K4XiYihomset}. $\Diamond$

		\begin{claim1}\label{K4XZhomset} Suppose that $G$ is $F_1$-free, and that $Y_{i+1}\cup Y_{i-1}\neq \es$. Then the following hold: If $Q$ is a component of $G[X_i]$, then for any $z\in Z$, $\{z\}$ is  either complete to $V(Q)$ or anticomplete to $V(Q)$. Further  if there is a vertex, say $z'\in Z$, such that $\{z'\}$ is not anticomplete to $V(Q)$, then $Q$ is $P_3$-free.\end{claim1}
		{\it Proof of \ref{K4XZhomset}}.~Let $z\in Z$ and let $y\in Y_{i+1}\cup Y_{i-1}$. Suppose to the contrary that there are vertices, say $x,x'\in V(Q)$ such that $xx',xz\in E(G)$ and $x'z\notin E(G)$.  We may assume that $y\in Y_{i+1}$. Then since $\{x',x,z,v_{i+2},y\}$ does not induce a $P_5$,  $xy,x'y\in E(G)$ (by \ref{AiBi+1homset}), and  then $\{x,z,v_3,y,v_1,x',v_2\}$ induces an $F_1$, a  contradiction. This proves the first assertion. Since $V(Q)$ is complete to $\{v_i\}$, the second assertion follows from the first assertion and from \ref{uvadjnbd}. This proves \ref{K4XZhomset}. $\Diamond$

\smallskip
Next we have the following crucial and useful theorem.

\begin{theorem}\label{HVNfreeper}
		Let $G$ be a connected $(P_5, K_5-e)$-free graph with $\omega(G)\geq t+3$ where $t\geq 1$. If $G$ is $\mathbb{H}_t$-free, then either $G$ is  the complement of a bipartite graph or $G$ has a clique cut-set.
	\end{theorem}
	\begin{proof}
		Let $G$ be a connected $(P_5,K_5-e)$-free graph which  has no clique cut-set. We show that $G$ is the complement of a bipartite graph. We may assume that $G$ is not a complete graph.  Since $\omega(G)\geq t+3$,  there are vertices, say $v_1,v_2,v_3, \ldots, v_{t+3}$ in $V(G)$ such that $\{v_1,v_2,v_3, \ldots, v_{t+3}\}$ induces a $K_{t+3}$, say $K$. Let $C := \{v_1,v_2,v_3\}$. Then, with respect to $C$, we define the sets $X,Y,Z$ and $L$ as above, and we use the properties  \ref{Zclq}-- \ref{Atcorel}. Note that $\{v_4, \ldots, v_{t+3}\}\subseteq Z$. Since $t\geq 1$,  $Z\neq \es$.    Moreover the following hold.
\begin{enumerate}[label=($\roman*$)]\itemsep=0pt
\item For any $y\in Y,$ since $K\cup \{y\}$ does not induce an $\mathbb{H}_t$ (by \ref{BZancom}), we have $Y=\es$.
\item If there  are adjacent vertices, say $x\in X$ and $z\in Z$, then  $\{x\}$ is anticomplete to $Z\sm \{z\}$  (by \ref{aZcard}), and then $K\cup \{x\}$ induces an $\mathbb{H}_t$; so  $X$ is anticomplete to $Z$.

\item By $(i)$ and $(ii)$, since $C$ is  not a clique cut-set separating $Z$ and $X$, we have $L\neq \es$.
\item For each $i\in \{1,2,3\}$, since $Z\cup \{v_{i+1}, v_{i-1}\}$ is not a clique cut-set separating $\{v_i\}$ and $L$ (by \ref{Zclq} and $(iii)$), we have $X_i\neq \es$, for each $i\in \{1,2,3\}$.
    \item Since $G$ is connected, and since  $Z$ is not a clique cut-set separating $\{v_1\}$ and the vertex-set of a component of $G[L]$ (by \ref{Zclq} and $(iii)$), the vertex-set of each component of $G[L]$ is not anticomplete to $X$. So   $X$ is complete to $L$, and $X_i$ is complete to $X_{i+1}$, for each $i\in \{1,2,3\}$ (by \ref{AiBi+1homset} and \ref{Atcorel}).
\end{enumerate}
   Now  from $(iii)$, $(iv)$ and $(v)$,  since $G[X_1\cup X_2\cup X_3\cup L]$ does not contain a $K_5-e$, it follows from \ref{4setsclique} that $X\cup L$ is a clique. Also $C\cup Z$ is a clique (by \ref{Zclq}). Thus from $(i)$, we conclude that $G$ is the complement of a bipartite graph. This proves \cref{HVNfreeper}.
	\end{proof}

\subsection{$(P_5, K_5-e)$-free graphs that contain one of $F_1$, $F_2$ or $F_3$}\label{GcontF123}

\subsubsection{$(P_5, K_5-e)$-free graphs that contain an $F_1$}\label{GcontF1}

Let $G$ be a connected ($P_5,K_5-e$)-free   graph which  has no clique cut-set. Suppose that $G$ contains an $F_1$ with vertices and edges as shown in Figure~\ref{fig-F123}. Let $C:=\{v_1,v_2,v_3\}$. Then, with respect to $C$, we define the sets $X$, $Y$, $L$ and $Z$ as in Section~\ref{genprop}, and we use the properties in Section~\ref{genprop}.  Clearly  $x_1\in X_1$, $y_2\in Y_2$, $y_3\in Y_3$ and $z\in Z$  so that $X_1$, $Y_2$, $Y_3$ and $Z$ are nonempty. Recall that $C\cup Z$ is a clique (by \ref{Zclq}), and that $Y$ is anticomplete to $Z$ (by \ref{BZancom}). Moreover the graph $G$ has some more properties which we give in three lemmas below.
\begin{lemma}\label{F1-prop} The following hold:

\vspace{-0.2cm}
\begin{lemmalist}\itemsep=0pt
\item\label{Y2Y3card} $Y_2=\{y_2\}$. Likewise, $Y_3=\{y_3\}$.
\item \label{a1b2b3Y1ancom} 	  $Y_1$ is anticomplete to $\{x_1\}\cup Y_2\cup Y_3$.
\item \label{a1Tancom}
 $L$ is anticomplete to $\{x_1\}\cup X_2\cup X_3$.
 \item \label{G1A2b2b3adj}
  Every vertex in $X_2\cup X_3$ has a neighbor in $\{y_2, y_3\}$.
  \item \label{A2Zhomset}
	  For $j\in \{2,3\}$, $\{z\}$ is either complete to $X_j$ or  anticomplete to $X_j$.
\end{lemmalist}
\end{lemma}
\begin{proof} $(i)$:~Suppose to the contrary that there is a vertex in $Y_2\sm \{y_2\}$, say $p$.  By \ref{BiBi+1bigcomancom}, $py_2\notin E(G)$. Then since $\{x_1,y_3,v_2,v_3,p\}$ or $\{x_1,y_3,p,v_3,z\}$ does not induce  a $P_5$  (by \ref{BZancom}), $px_1\in E(G)$ and $py_3\in E(G)$. Then $\{x_1,y_3,v_1,p,y_2\}$ induces a $K_5-e$, a contradiction. So \cref{Y2Y3card} holds. $\diamond$

\smallskip
\no{$(ii)$}:~Suppose to the contrary that there is a vertex in $Y_1$, say $p$, which has a neighbor in $\{x_1,y_2,y_3\}$ (by \cref{Y2Y3card}). Since $\{v_1,x_1,y_2,y_3,p\}$ does not induce a $K_5-e$, $p$ has a nonneighbor in $\{x_1,y_2,y_3\}$.  Now  if  $px_1\in E(G)$, then  we may assume (up to symmetry) that $py_2\notin E(G)$, and then $\{y_2,x_1,p,v_2,z\}$ induces a $P_5$, a contradiction; so  $px_1\notin E(G)$.  Then  we may assume (up to symmetry) that $py_2\in E(G)$, and then $\{x_1,y_2,p,v_2,z\}$ induces a $P_5$, a contradiction. So \cref{a1b2b3Y1ancom} holds. $\diamond$

\smallskip
\no{$(iii)$}:~If there is a vertex, say $t\in L$, such that $tx_1\in E(G)$, then $\{x_1,y_2,y_3,t,v_1\}$ induces a $K_5-e$ (by \ref{AiBi+1homset}); so $L$ is anticomplete to $\{x_1\}$.  This implies that $L$ is anticomplete to $X_2\cup X_3$ (by \ref{Atcorel}:(a)).
This proves \cref{a1Tancom}. $\diamond$

\smallskip
\no{$(iv)$}:~Suppose to the contrary that there is a vertex in $X_2\cup X_3$, say $p$, which is anticomplete to $\{y_2,y_3\}$. We may assume, up to symmetry, that  $p\in X_2$. Now since $\{x_1,y_2,v_3,v_2,p\}$ does not induce a $P_5$,   $px_1\in E(G)$, and then  one of $\{y_3,x_1,p,z,v_3\}$ or $\{p,x_1,y_2,v_3,z\}$ induces a $P_5$, a contradiction. So \cref{G1A2b2b3adj} holds. $\diamond$

\smallskip
\no{$(v)$}:~We prove the assertion for $j = 2$. Suppose to the contrary that there are vertices, say $p, q\in X_2$ such that $pz\in E(G)$ and $qz \notin E(G)$. Since  one of $\{x_1,y_2,v_3,v_2,q\}$ or $\{y_2,x_1,q,v_2,z\}$ does not induce a $P_5$, $qx_1,qy_2\in E(G)$. Then since $\{v_1,x_1,y_2,y_3, q\}$ does not induce a $K_5-e$,  $qy_3\notin E(G)$. Then since $\{q,x_1,v_1,z,p\}$ does not induce a $P_5$, either $pq\in E(G)$ or $px_1\in E(G)$. Now  if $pq\in E(G)$, then  $py_3\notin E(G)$ and $px_1\in E(G)$ (by \ref{AiBi+1homset}), and then $\{y_3,x_1,p,z,v_3\}$ induces a $P_5$, a contradiction. So, we may assume that $pq\notin E(G)$, and hence $px_1\in E(G)$. Then $\{q,x_1,p,z,v_3\}$ induces a $P_5$, a contradiction. So \cref{A2Zhomset} holds.
\end{proof}

\begin{lemma}\label{A1clqX1Y2Y3com}
 The set $X_1\cup Y_2\cup Y_3$ is a clique.
\end{lemma}
\begin{proof}
By \ref{AiBi+1homset} and \cref{Y2Y3card}, it is enough to show that $X_1$ is a clique.  Suppose to the contrary that there are nonadjacent vertices in $X_1$, say $p$ and $q$. Since $\{y_2,y_3,x_1,p,q\}$ does not induce a $K_5-e$  (by \ref{AiBi+1homset}), we may assume that $px_1\notin E(G)$, and hence $X_1\sm (N(x_1)\cup \{x_1\})\neq \es$. We  let $X_1':=X_1\sm (N(x_1)\cup \{x_1\})$ and let $L':=\{t\in L \mid t \mbox{ has a neighbor in } X_1'\}$. We will show that $X_1'\cup L'$ is anticomplete to $V(G)\sm (X_1'\cup L'\cup \{v_1\})$, and thus $v_1$ is a cut-vertex in $G$. Recall that $V(G)\sm (X_1'\cup L'\cup \{v_1\})= \{v_2,v_3\}\cup (X\sm X_1')\cup Y\cup Z \cup (L\sm L')$. Let $x\in X_1'$ be  arbitrary.  Then since one of $\{x,z,v_3,y_2,x_1\}$, $\{x_1,y_2,x,z,v_2\}$ does not induce a $P_5$,    $xz\notin E(G)$. Moreover, we have the following.
\begin{enumerate}[leftmargin=0.7cm, label={($\arabic*$)}] \itemsep=0pt

\item \textit{$\{x\}$ is anticomplete to $Y\cup Z$}:~If $xy_2\in E(G)$, then    $\{y_2,y_3,v_1,x,x_1\}$ induces a $K_5-e$ or  $\{x,y_2,y_3,v_2,z\}$ induces a $P_5$; so $xy_2\notin E(G)$. Likewise,   $xy_3\notin E(G)$. Hence for any $w\in Y_1\cup Z$, since   $\{x,w,v_2,y_3,y_2\}$ does not induce a $P_5$ (by \cref{a1b2b3Y1ancom}), $\{x\}$ is anticomplete to $Y_1\cup Z$.  So  $\{x\}$ is anticomplete to $Y\cup Z$ (by \cref{Y2Y3card}).

\item \textit{$\{x\}$ is anticomplete to $X_2\cup X_3$}:~Suppose that $x$ has a neighbor in $X_2$, say $a_2$. Since $\{x,a_2,v_2,v_3,y_2\}$ does not induce a $P_5$,   $a_2y_2\in E(G)$.  If $a_2z\notin E(G)$, then $\{x,a_2,y_2,v_3,z\}$ induces a $P_5$; so we may assume that $a_2z\in E(G)$. Now  since $\{v_1,x_1,y_2,y_3,a_2\}$ does not induce a $K_5-e$ or  $\{x_1,y_3,a_2,z,v_3\}$ does not induce a $P_5$, we have $x_1a_2,a_2y_3\notin E(G)$, and then $\{x,a_2,v_2,y_3,x_1\}$ induces a $P_5$; so $\{x\}$ is anticomplete to $X_2$. Likewise, $\{x\}$ is anticomplete to $X_3$.

    \item Since $y_2x_1\in E(G)$, it follows from \ref{AiBi+1homset} that $\{y_2\}$ is complete to $X_1\sm X_1'$. So it follows from $(1)$ and  \ref{AiBi+1homset} that $X_1'$ is anticomplete to $X_1\sm X_1'$. Further   $X_1'\cup L'$ is anticomplete to $L\sm L'$ (by the definition of $L'$ and by \ref{AiBi+1homset}).

\item  By \ref{AiBi+1homset} and $(1)$, $L'$ is anticomplete to $Y_2\cup Y_3$. So $L'$ is anticomplete to $(X_1\cap N(x_1))\cup \{x_1\}$ (by \ref{AiBi+1homset}). Hence $X_2\cup X_3=\es$ (by \ref{Atcorel}:(a)). If $t\in L'$ has a neighbor in $Y_1\cup Z$, say $w$, then for any  neighbor of $t$ in $X_1'$, say $a'$, we see that $\{a',t,w,v_3,y_2\}$ induces a $P_5$ (by \cref{a1b2b3Y1ancom}); so $L'$ is anticomplete to $Y_1\cup Z$.
\end{enumerate}
Now since $x\in X_1'$  is arbitrary, from above arguments, we conclude that $v_1$ is a cut-vertex of $G$ (by \cref{a1Tancom}), a contradiction. This proves \cref{A1clqX1Y2Y3com}.
\end{proof}

\begin{lemma}
 \label{Temp}
	 The set $L$ is  an empty set.
 \end{lemma}
\begin{proof}Suppose to the contrary that $L\neq \es$. First we assume that $N(y_2)\cap L=\es$ and $N(y_3)\cap L=\es$. Then $L$ is anticomplete  to $X$ (by  \cref{a1Tancom},  \cref{A1clqX1Y2Y3com}, and by \ref{AiBi+1homset}). Since $Z$ is a clique (by \ref{Zclq}) and since $Z$ is not a clique cut-set of $G$ separating $L$ and $C$ (by \cref{Y2Y3card}), $L$ is not anticomplete to $Y_1$, and so there are adjacent vertices, say $y\in Y_1$ and $t\in L$. Then $\{y_2,y_3,v_2,y,t\}$ induces a $P_5$ (by \cref{a1b2b3Y1ancom}), a contradiction. So  we may assume, up to symmetry, that
  $N(y_2)\cap L\neq \es$, and let $t\in N(y_2)\cap L$. Then $tx_1\notin E(G)$ (by \cref{a1Tancom}). We claim that $C\cup (Z\sm \{z\})$ is a clique cut-set of $G$ (using \ref{Zclq}) separating $\{z\}$ and $L$. It is enough prove that $\{z\}$ is anticomplete to $X\cup Y\cup L$. Clearly, $\{z\}$ is anticomplete to $Y$ (by \ref{BZancom}).
   Next we show that  $\{z\}$ is anticomplete to $L$. Suppose to the contrary that $z$ has a neighbor in $L$, say $p$. Then since $\{x_1,y_2,p,z,v_2\}$ does not induce a $P_5$ (by \cref{a1Tancom}),  $p\neq t$ and $y_2p\notin E(G)$. So $tp\notin E(G)$ (by \ref{AiBi+1homset}), and then  $\{t,y_2,v_1,z,p\}$ induces a $P_5$, a contradiction. So $\{z\}$ is anticomplete to $L$.
  Finally, we show that  $\{z\}$ is anticomplete to $X$. Suppose to the contrary that $z$ has a neighbor in $X$, say $x$. If $x\in X_1$, then $\{t,y_2,x,z,v_2\}$ induces a $P_5$  or $\{x,y_2,y_3,v_1,t\}$ induces a $K_5-e$ (by $(i)$, \ref{AiBi+1homset} and \cref{A1clqX1Y2Y3com}), a contradiction. If $x\in X_2$, then since $\{t,y_2,y_3,v_2,z\}$ does not induce a $P_5$,  we have $ty_3\in E(G)$, and then $\{t,y_3,v_1,z,x\}$ or   $\{t,y_3,x,z,v_3\}$    induces a $P_5$ (by \cref{a1Tancom}), a contradiction. We get a similar contradiction  when $x\in X_3$. These contradictions show that $\{z\}$ is anticomplete to $X$.
 Thus, $C\cup (Z\sm \{z\})$ is a clique cut-set of $G$ separating $\{z\}$ and the rest of the vertices, a contradiction. This proves \cref{Temp}.
\end{proof}

Now we prove the main theorem of this section, and is given below.

\begin{theorem}\label{cont-HG}
		Let $G$ be a connected ($P_5,K_5-e$)-free   graph. If $G$ contains  an   $F_1$, then either $G$ is  the complement of a bipartite graph or $G$ has a clique cut-set or  $\chi(G)\leq 5$.
	\end{theorem}
\begin{proof}
	Let $G$ be a connected ($P_5,K_5-e$)-free   graph. Suppose that $G$ contains an $F_1$ with vertices and edges as shown in Figure~\ref{fig-F123}. Let $C:=\{v_1,v_2,v_3\}$. Then, with respect to $C$, we define the sets $X$, $Y$, $L$ and $Z$ as in Section~\ref{genprop}, and we use the properties in Section~\ref{genprop}.  Clearly  $x_1\in X_1$, $y_2\in Y_2$, $y_3\in Y_3$ and $z\in Z$  so that $X_1$, $Y_2$, $Y_3$ and $Z$ are nonempty. We may assume that $G$ has no clique cut-set, and that $G$ is not the complement of a bipartite graph.  We also use \Cref{F1-prop,A1clqX1Y2Y3com,Temp}. By \cref{Temp}, $L=\es$.  Recall that $C\cup Z$ is a clique (by \ref{Zclq}), and that $Y$ is anticomplete to $Z$ (by \ref{BZancom}). We show that $\chi(G)\leq 5$   using a sequence of claims given below.

\begin{claim}\label{A1card}
			 $X_1=\{x_1\}$.
		\end{claim}
		\no{\it Proof of  \ref{A1card}}.~Suppose not, and let  $X_1\sm\{x_1\}\neq \es$, say $x_1'\in X_1\sm\{x_1\}$.
By  \cref{A1clqX1Y2Y3com}, $X_1\cup Y_2\cup Y_3$ is a clique. For any $x\in X_2\cup X_3$, since $G[\{v_1,x_1,x_1',y_2,y_3,x\}]$ does not contain a $K_5-e$ (by \cref{G1A2b2b3adj} and \ref{AiBi+1homset}), we see that $X_2\cup X_3$ is anticomplete to $X_1$. Also,  since  $(X_1\sm \{x_1\})\cup Y_2\cup Y_3\cup \{v_1\}$ is not a clique cut-set of $G$ (by \cref{a1b2b3Y1ancom}) separating $\{x_1\}$ and $Z$,   there is a vertex in $Z$, say $z'$, such that $x_1z'\in E(G)$.
			  Now, we have the following:
			\begin{enumerate}[label={($\roman*$)}]\itemsep=0pt
				\item For any $p\in X_1\sm \{x_1\}$, since $\{x_1,p,v_1,z',y_2\}$ does not induce a $K_5-e$ (by \ref{BZancom}), we have $X_1\sm \{x_1\}$ is anticomplete to $\{z'\}$. In particular, $x_1'z'\notin E(G)$.
				\item Next, we claim that $X_2\cup X_3=\es$. Suppose not, and let $q\in X_2$. Then since $\{x_1',x_1,z',v_2,q\}$ does not induce a $P_5$ (by $(i)$), $qz'\in E(G)$. Then $qz\notin E(G)$ (by  \ref{aZcard}). But then one of $\{v_3,z',q,y_3,x_1'\}$ or $\{x_1,y_2,q,v_2,z\}$ induces a $P_5$ (by \cref{G1A2b2b3adj} and $(i)$), a contradiction. So $X_2=\es$. Likewise, $X_3=\es$.
				\item Finally, we claim that $Y_1=\es$. Suppose not. Then since $\{v_2,v_3\}$ is not a clique cut-set of $Y_1$ and the rest of the vertices  (by $(ii)$ and \cref{a1b2b3Y1ancom}),   $X_1\sm \{x_1\}$ is not anticomplete to $Y_1$. So there are adjacent vertices, say $p\in X_1\sm \{x_1\}$ and $q\in Y_1$.  Then  $\{y_2,p,q,v_2,z'\}$ induces a $P_5$ (by \ref{BZancom} and $(i)$), a contradiction. So  $Y_1=\es$.
			\end{enumerate}
		Then, by above arguments, $V(G)$ can be partitioned into two  cliques, namely, $X_1\cup Y_2\cup Y_3$ and $C\cup Z$. Thus $G$ is the complement of a bipartite graph, a contradiction.  So   $X_1=\{x_1\}$.
		This proves \ref{A1card}. $\Diamond$

\begin{claim}\label{A2A3Y1ancom}
	$X\cup Y = \{x_1,y_2,y_3\} \cup X_2\cup X_3$.
\end{claim}
\no{\it Proof of \ref{A2A3Y1ancom}}.~By \cref{Y2Y3card} and \ref{A1card}, it is enough to show that $Y_1=\es$. First we show that $X_2\cup X_3$ is anticomplete to $Y_1$. Suppose to the contrary that there are adjacent vertices, say $p\in X_2\cup X_3$ and $q\in Y_1$. We may assume, up to symmetry, that $p\in X_2$. From \ref{BZancom} and \cref{a1b2b3Y1ancom}, $\{q\}$ is anticomplete to  $\{z,x_1,y_2,y_3\}$. Then since $\{y_3,v_1,v_3,q,p\}$ and $\{x_1,v_1,v_3,q,p\}$ do not induce $P_5$'s, we have $py_3,px_1\in E(G)$.  Then since $\{y_3,p,q,v_3,z\}$ does not induce a $P_5$, we have $pz\in E(G)$, and since $\{v_1,x_1,y_2,y_3,p\}$ does not induce a $K_5-e$, we have $py_2\notin E(G)$. Now  $\{q,p,z,v_1,y_2\}$ induces a $P_5$, a contradiction.  So $X_2\cup X_3$ is anticomplete to $Y_1$. Hence from \ref{BZancom},  \ref{A1card} and \cref{a1b2b3Y1ancom}, we conclude that $Y_1$ is anticomplete to $X\cup Y_2\cup Y_3\cup Z\cup L$. Now since  $\{v_2,v_3\}$ is not a clique cut-set of $G$ separating $Y_1$ and the rest of the vertices in $G$, we have $Y_1=\es$.  This proves \ref{A2A3Y1ancom}. $\Diamond$

\begin{claim}\label{G1Zcard}
		 $|Z\sm \{z\}|\leq 1$.
	\end{claim}
	\no{\it Proof of \ref{G1Zcard}}.~Suppose not. Then there is a vertex in $Z\sm \{z\}$, say $z'$, such that $x_1z'\notin E(G)$ (by \ref{aZcard}). By \ref{A2A3Y1ancom}, since $C\cup (Z\sm\{z'\})$ is not a clique cut-set of $G$ separating $\{z'\}$ and the rest of the vertices, we may assume that $z'$ has a neighbor in $X_2$, say $q$. Then $qz\notin E(G)$ (by \ref{aZcard}). Then as in the proof of \cref{A2Zhomset}, we have $qx_1,qy_2\in E(G)$, and $qy_3\notin E(G)$. Then $\{y_3,x_1,q,z',v_3\}$ induces a $P_5$ (by \ref{BZancom}), a contradiction. This proves \ref{G1Zcard}. $\Diamond$

\begin{claim}\label{G1Y2A2znbdhomset}
	 For $j\in \{2,3\}$, $G[X_j]$ is the union of $K_2$'s and $K_1$'s.
\end{claim}
	\no{\it Proof of \ref{G1Y2A2znbdhomset}}.~We prove the claim for $j = 2$.  Let $Q$ be a  component of $G[X_2]$. It is enough to show that $V(Q)$ induces a $(P_3,K_3)$-free graph. By \cref{A2Zhomset}, we have either $V(Q)$ is complete to $\{z\}$ or $V(Q)$ is anticomplete to $\{z\}$. First suppose that  $V(Q)$ is complete to $\{z\}$. Then since $V(Q)$ is complete to $\{v_2,z\}$,  by \ref{uvadjnbd}, $Q$ is $P_3$-free.
 Also, since $G[\{z,y_2,y_3\}\cup V(Q)]$ does not contain a $K_5-e$ (by  \ref{AiBi+1homset} and \cref{G1A2b2b3adj}), $Q$ is $K_3$-free, and we are done. So we may assume that $V(Q)$ is anticomplete to $\{z\}$. Then as in the proof of \cref{A2Zhomset},  we see that   $V(Q)$ is complete to $\{x_1,y_2\}$, by using \ref{AiBi+1homset}. Thus  $Q$ is   $P_3$-free (by \ref{uvadjnbd}), and   since $G[\{x_1,v_2\}\cup V(Q)]$ does not contain a $K_5-e$, $Q$ is  $K_3$-free.
	 	This proves \ref{G1Y2A2znbdhomset}. $\Diamond$

\medskip
By \ref{G1Y2A2znbdhomset}, for $j\in \{2,3\}$, we let $X_j:=A_j\cup B_j$, where $A_j$ and $B_j$ are stable sets such that $B_j$ is maximal.  Then we have the following:

\begin{claim}\label{G1colstr}
 $A_2\cup \{y_3\}$ is a stable set. Likewise, $A_3\cup \{y_2\}$ is a stable set.
\end{claim}
	\no{\it Proof of \ref{G1colstr}}.~Suppose to the contrary that there is a vertex in  $A_2$, say $p$, such that $py_3\in E(G)$. Then by our definition of $A_2$, since $B_2$ is a maximal stable set, there is a vertex in $B_2$, say $q$, such that $pq\in E(G)$. Since $\{p,q,v_2,y_3,z\}$ does not induce a $K_5-e$ (by \ref{AiBi+1homset}), we have $pz\notin E(G)$ (by \cref{A2Zhomset}). Then since $\{p,y_3,y_2,v_3,z\}$ does not induce a $P_5$, we have $py_2\in E(G)$, and then since $\{v_1,x_1,y_2,y_3,p\}$ does not induce a $K_5-e$, we have $px_1\notin E(G)$. But then $\{x_1,y_2,p,v_2,z\}$ induces a $P_5$, a contradiction. This proves \ref{G1colstr}. $\Diamond$
	
\medskip
By \ref{A2A3Y1ancom}, we conclude that   $V(G)= C\cup Z \cup A_2\cup B_2\cup A_3\cup B_3\cup \{x_1,y_2,y_3\}$.  Since $C\cup (Z\sm \{z\})$ is not a clique cut-set of $G$ separating $\{z\}$ and the rest of the vertices in $G$,    $z$ has a neighbor in $X_2\cup X_3$. So  we may assume that $\{z\}$ is complete to $X_2$ (by \cref{A2Zhomset}), and hence  $X_2$ is anticomplete to $Z\sm \{z\}$ (by \ref{aZcard}). Also $Z\sm \{z\}$ is anticomplete to $\{y_3\}$ (by \ref{BZancom}). Now  by using \ref{G1Zcard} and \ref{G1colstr}, we define the following stable sets: $S_1:=A_2\cup (Z\sm \{z\})\cup \{y_3\}$, $S_2:=B_2 \cup \{v_3\}$, $S_3:=\{x_1, z\}$, $S_4:=A_3\cup \{y_2, v_2\}$ and $S_5:=B_3 \cup \{v_1\}$. Clearly $V(G)=\cup_{j=1}^5S_j$, and hence $\chi(G)\leq 5$. This proves \cref{cont-HG}.
\end{proof}

\subsubsection{$(P_5, K_5-e, F_1)$-free graphs that contain an $F_2$}\label{GcontF2}
Let $G$ be a connected ($P_5,K_5-e, F_1$)-free   graph which has no clique cut-set. Suppose that $G$ contains an   $F_2$ with vertices and edges as shown in Figure~\ref{fig-F123}. Let $C:=\{v_1,v_2,v_3\}$. Then  with respect to $C$, we define the sets $X$, $Y$, $Z$ and $L$ as in Section~\ref{genprop}, and we use the properties in Section~\ref{genprop}.  Clearly  $y_2\in Y_2$, $y_3\in Y_3$ and $z_1,z_2\in Z$  so that $Y_2$, $Y_3$ and $Z$ are nonempty.   Recall that $C\cup Z$ is a clique (by \ref{Zclq}), and that $Y$ is anticomplete to $Z$ (by \ref{BZancom}). Moreover, the graph $G$ has some more properties which we give below in two lemmas.

\begin{lemma}\label{lemma-F2}The following hold:

\vspace{-0.2cm}
\begin{lemmalist}
\item \label{G2A1B2B3ancom}
	$X_1$ is anticomplete to $Y_2\cup Y_3$. Likewise, if $Y_1\neq \es$, then for each $i\in \{1,2,3\}$, $X_i$ is anticomplete to $Y_{i+1}\cup Y_{i-1}$.
\item \label{G2X1Z} $X_1$ is anticomplete to $Z$.
\item \label{G2xy2y3com} If a vertex in $X$ has a neighbor in $Z$, then it is adjacent to both $y_2$ and $y_3$.
\item \label{G2A1A2ancom}
	$X_1$ is anticomplete to $X_2\cup X_3$.
	Likewise, if $Y_1\neq \es$, then for each $i\in \{1,2,3\}$,  $X_i$ is anticomplete to $X_{i+1}\cup X_{i-1}$.
\item \label{G2znbdXicard}
	 For $j\in \{2,3\}$, $|z\in Z: N(z)\cap X_j\neq \es|\leq 1$.
\end{lemmalist}
\end{lemma}
\begin{proof}$(i)$:~Suppose to the contrary that there are adjacent vertices, say $x\in X_1$ and $y\in Y_2\cup Y_3$. We may assume (up to symmetry) that $y\in Y_2$, and  we may assume that $xz_1\notin E(G)$ (by \ref{aZcard}). Then since one of $\{x,y,v_3,v_2,y_3\}$ or $\{x,y,y_3,v_2,z_1\}$ does  not induce a $P_5$, we have  $xy_3,yy_3\in E(G)$, and then $\{v_1,v_2,v_3,x,y,y_3,z_1\}$ induces an $F_1$, a contradiction. This proves \cref{G2A1B2B3ancom}. $\diamond$

\smallskip
\no{$(ii)$}:~If there are adjacent vertices, say $x\in X_1$ and $z\in Z$, then $\{x,z,v_3,y_2,y_3\} $ induces a $P_5$ (by \ref{BZancom}  and \cref{G2A1B2B3ancom}). This proves \cref{G2X1Z}. $\diamond$

\smallskip
\no{$(iii)$}:~Let $x$ be a vertex in $X$ which has a neighbor in $Z$, say $z$. We may assume that $z\neq z_1$ (by \ref{aZcard}).  Clearly $x\notin X_1$ (by \cref{G2X1Z}), and we may assume, up to symmetry, that $x\in X_2$. Then $xz_1\notin E(G)$ (by \ref{aZcard}). Then since one of $\{x,z,v_3,y_2,y_3\}$   or $\{y_2,y_3,x,z,z_1\}$  does not induce a $P_5$ (by \ref{Zclq}), we have $xy_2,xy_3\in E(G)$. This proves \cref{G2xy2y3com}. $\diamond$

\smallskip
\no{($iv)$}:~Suppose to the contrary that there are adjacent vertices, say $x\in X_1$ and $x'\in X_2\cup X_3$. We may assume (up to symmetry) that $x'\in X_2$. Also  we may assume that $x'z_1\notin E(G)$ (by \ref{aZcard}). Moreover, $xz_1,xy_2,xy_3\notin E(G)$ (by \cref{G2A1B2B3ancom} and \cref{G2X1Z}). Then one of  $\{x,x',v_2,v_3,y_2\}$ or $\{x,x',y_2,v_3,z_1\}$ induces a $P_5$, a contradiction. This proves \cref{G2A1A2ancom}. $\diamond$

\smallskip
\no{$(v)$}:~We will show for $j=2$. Suppose that the assertion is not true. Then there are vertices,  say $x$, $x'$ in $X_2$, and  $z$, $z'$ in $Z$ such that $xz,x'z'\in E(G)$. Then $xz',x'z\notin E(G)$ (by \ref{aZcard}), and  $xy_3,x'y_3\in E(G)$ (by \cref{G2xy2y3com}).  Now since $\{x,y_3,x',z',v_3\}$ does not induce a $P_5$, we have $xx'\in E(G)$. But then $\{z,z',v_2,v_3,x,x',y_3\}$ induces an $F_1$, a contradiction. This proves \cref{G2znbdXicard}.
\end{proof}

\begin{lemma}\label{G2X2Y3com} For $i\in \{1,2,3\}$, suppose that  there are vertices, say $p\in X_i$, $y\in Y_i$, $z\in Z$ and $t\in L$ such that $pt,zt\in E(G)$ and $pz\notin E(G)$. Then  for $j\in \{i+1,i-1\}$, the following hold:  (i) $\{p,t\}$ is complete to $Y_j$. (ii) If $y$ has a neighbor in $Y_j$, then $\{y\}$  is complete to $Y_j\cup \{p\}$.\end{lemma}
\begin{proof}We will show for $j=i+1$. Since for any $y'\in Y_{i+1}$, $\{p,t,z,v_{i-1},y'\}$ does not induce a $P_5$, $\{p,t\}$ is complete to $Y_{i+1}$ (by \ref{AiBi+1homset}). This proves $(i)$. To prove $(ii)$, we pick neighbor of $y$ in $Y_{i+1}$, say $r$.  Then since $\{p,r,y,v_{i+1},z\}$ does not induce a $P_5$,  $py\in E(G)$. Hence, for any $u\in Y_{i+1}$, $\{u,p,y,v_{i+1},z\}$ does not induce a $P_5$, $\{y\}$  is complete to $Y_{i+1}$. This proves \cref{G2X2Y3com}. \end{proof}

\begin{lemma}\label{G2LYcorel}
	For any $t\in L$, we have either $ty_2\in E(G)$ or $ty_3\in E(G)$.
\end{lemma}
\begin{proof}Suppose not. Let $t\in L$ be such that $ty_2, ty_3\notin E(G)$. Let $Q$  be the component of $G[L]$ such that $t\in V(Q)$.  Then $V(Q)$ is anticomplete to $\{y_2,y_3\}$ (by \ref{AiBi+1homset}). We claim that $V(Q)$ is   anticomplete to $X_2\cup X_3\cup Y\cup Z$. If not, then there are adjacent vertices, say $a\in V(Q)$ and (up to symmetry)  $b\in Z\cup X_2\cup Y_1$. If $b\in Z$, then $\{a,b,v_3,y_2,y_3\}$ induces a $P_5$; so $V(Q)$ is anticomplete to $Z$. If $b\in X_2\cup  Y_1$,  then we may assume that $bz_1\notin E(G)$ (by \ref{BZancom} and \ref{aZcard}), and then $\{a,b,v_2,v_1,y_2\}$  or $\{a,b,y_2,v_1,z_1\}$ induces a $P_5$; so $V(Q)$ is   anticomplete to $X_2\cup X_3\cup Y\cup Z$.    Now since $G$ is connected,  $V(Q)$ is not anticomplete to $X_1$, and so there are adjacent vertices, say $q\in V(Q)$ and $x\in X_1$. Then by \ref{Atcorel} and \cref{G2A1A2ancom}, $X_2\cup X_3=\es$.
Since $X_1$ is anticomplete to $Y_2\cup Y_3\cup Z$ (by \cref{G2A1B2B3ancom} and \cref{G2X1Z}), and since  $C\cup Z$ is not a clique cut-set of $G$ (by \ref{Zclq} and \ref{AiBi+1homset}) separating $X_1\cup V(Q)$ and $Y_2\cup Y_3$, $X_1$ is not anticomplete to $Y_1$. So $Y_1\neq \es$, and let $y\in Y_1$. But then one of $\{q,x,v_1,v_2,y\}$ or  $\{q,x,y,v_2,z_1\}$ induces a $P_5$ (by \ref{BZancom} and \cref{G2X1Z}), a contradiction. This proves \cref{G2LYcorel}.
  \end{proof}

  Now we prove the main theorem of this section, and is given below.

\begin{theorem}\label{cont-H2}
		Let $G$ be a connected ($P_5,K_5-e$)-free  graph. If $G$ contains  an $F_2$, then either $G$ is the complement of a bipartite graph or $G$ has a clique cut-set or $\chi(G)\leq 5$.
	\end{theorem}
\begin{proof}
Let $G$ be a connected ($P_5,K_5-e$)-free   graph. Suppose that $G$ contains an   $F_2$ with vertices and edges as shown in Figure~\ref{fig-F123}. Let $C:=\{v_1,v_2,v_3\}$. Then, with respect to $C$, we define the sets $X$, $Y$, $Z$ and $L$ as in Section~\ref{genprop}, and use the properties in Section~\ref{genprop}.  Clearly, $y_2\in Y_2$, $y_3\in Y_3$ and $z_1,z_2\in Z$, so that $Y_2$, $Y_3$ and $Z$ are nonempty.  Recall that $C\cup Z$ is a clique (by \ref{Zclq}), and that $Y$ is anticomplete to $Z$ (by \ref{BZancom}). We may assume that $G$ has no clique cut-set, and that $G$ is not the complement of a bipartite graph.   From \cref{cont-HG}, we may assume that $G$ is $F_1$-free,   and we use \cref{lemma-F2} and \cref{G2LYcorel}.   We show that $\chi(G)\leq 5$ by using a sequence of claims given below.

  \begin{claim}\label{G2X1L}
  $X_1$ is anticomplete to $L$.
  \end{claim}
\no{\it Proof of \ref{G2X1L}}. This follows from \cref{G2A1B2B3ancom} and \cref{G2LYcorel}, and \ref{AiBi+1homset}.  $\Diamond$

\begin{claim}\label{G2LZcom}
	$L$ is complete to $Z$, and so $L$ induces a bipartite graph.
\end{claim}
\no{\it Proof of \ref{G2LZcom}}.~Suppose to the contrary that there are nonadjacent vertices, say $t\in L$ and $z\in Z$. Let $Q$ be the component of $G[L]$ such that $t\in V(Q)$. We may assume that $ty_2\in E(G)$ (by \cref{G2LYcorel}). Then since $\{t,y_2,y_3,v_2,z\}$ does not induce a $P_5$ (by \ref{BZancom}), we have $ty_3\in E(G)$.  So $V(Q)$ is complete to $\{y_2,y_3\}$ (by \ref{AiBi+1homset}),  and hence  $V(Q)$ is a clique (by \ref{uvadjnbd}). Moreover, we have the following:
\begin{enumerate}[leftmargin=0.78cm, label={($\alph*$)}]\itemsep=0pt
	\item Since for any $y\in Y_3$, one of $\{t,y_2,y,v_2,z\}$ or $\{y,v_2,v_3,y_2,t\}$ does not induce a  $P_5$,  $Y_3$ is complete to $\{t,y_2\}$. By using similar arguments, we see that $Y_2$ is complete to $Y_3\cup \{t\}$, and $Y_1$ is complete to $Y_2\cup Y_3\cup \{t\}$.

\item  If $X_1\neq\es$, then since $X_1$ is anticomplete to $(X\cup Y\cup Z\cup L)\sm (X_1\cup Y_1)$ (by \cref{G2X1Z}, \cref{G2A1B2B3ancom}, \cref{G2A1A2ancom} and \ref{G2X1L}) and since $\{v_1\}$ is not a cut-vertex of $G$ separating $X_1$ and rest of the vertices, there are adjacent vertices, say $x\in X_1$ and $y\in Y_1$, and then  $\{t,y,x,v_1,z\}$ induces a $P_5$ (by \ref{BZancom} and $(a)$), a contradiction. So  $X_1=\es$.
	\item Now we will show that $X$ is complete to $Y_2\cup Y_3\cup \{t\}$. Let $x\in X_2$. If $xz\notin E(G)$, then by using a similar argument as in $(a)$, we see that $\{x\}$ is complete to $Y_2\cup \{t\}$, and if $xz\in E(G)$, then by using a similar argument as in \cref{G2xy2y3com},  we see that $\{x\}$ is complete to $Y_2\cup \{t\}$.   Also, since for any   $y\in Y_3$, $\{x,t,y,v_1,v_3\}$ does not induce a $P_5$,   $\{x\}$ is complete to $Y_3$. Since $x$ is arbitrary, $X_2$ is complete to $Y_2\cup Y_3\cup \{t\}$.  Likewise, $X_3$ is complete to $Y_2\cup Y_3\cup \{t\}$. So $X$ is complete to $Y_2\cup Y_3\cup \{t\}$ (by $(b)$).
	\item From $(a)$, $(b)$, $(c)$ and \ref{AiBi+1homset}, we conclude that $V(Q)$ is complete to $X\cup Y$.
		\item Next we will show that $L\sm V(Q)=\es$. Suppose not. Let $q\in L\sm V(Q)$. Since for any $t'\in L\sm V(Q)$ and for $j\in \{2,3\}$, $\{v_j,z,t',y_j,t\}$ does not induce a $P_5$ (by \cref{G2LYcorel}),  $L\sm V(Q)$ is anticomplete to $\{z\}$. Likewise, $V(Q)$ is anticomplete to $\{z\}$. Then as in $(c)$, the component that contains $q$ is complete to $X\cup Y_2\cup Y_3$. Since $C\cup (Z\sm \{z\})$ is not a clique cut-set of $G$ separating $\{z\}$ and the rest of the vertices, there is a vertex in $X$, say $x$ such that $xz\in E(G)$. Then  $\{x,y_2,y_3,t,q\}$ induces a $K_5-e$, a contradiction. Hence $L\sm V(Q)=\es$.
\end{enumerate}
Now if $X_2\cup X_3\cup Y_1\neq \es$, then since $Y_2, Y_3$ and $L$ are nonempty, by above arguments and \ref{4setsclique}, $X\cup Y\cup L$ is a clique, and hence $G$ is the complement of a bipartite graph (by \ref{Zclq}), a contradiction; so we may assume that 	$X_2\cup X_3\cup Y_1= \es$.  Then since $C\cup (Z\sm \{z\})$ is not a clique cut-set   separating $\{z\}$ from the rest of the vertices of $G$,   there is a vertex in $V(Q)\sm \{t\}$, say $t'$, such that $t'z\in E(G)$. Now since $Y_2, Y_3, \{t\}$ and $V(Q)\sm \{t\}$ are nonempty, again by above arguments and \ref{4setsclique}, $X\cup Y \cup L$ is a clique.
   So $V(G)$ can be partitioned into two cliques, namely $C\cup Z$ (by \ref{Zclq}) and $X\cup Y\cup L$, and hence $G$ is the complement of a bipartite graph, a contradiction. This proves the first assertion.

Now  since $L$ is complete to $\{z_1,z_2\}$,  $L$ induces a $P_3$-free graph (by \ref{uvadjnbd}). Moreover, by using \cref{G2LYcorel}, \ref{BZancom}, \ref{AiBi+1homset} and \ref{uvnonadjnbd}, we see that each component of $G[L]$ is $K_3$-free. So $L$ induces a ($P_3, K_3$)-free graph, which is a bipartite graph. This proves \ref{G2LZcom}.
 $\Diamond$


\begin{claim}\label{G2Zz1z2}
 $Z=\{z_1,z_2\}$.
\end{claim}
\no{\it Proof of \ref{G2Zz1z2}}.~If $L\neq \es$ then  $Z=\{z_1,z_2\}$ (by \ref{lZcard} and \ref{G2LZcom}). So we may assume that $L=\es$. Suppose to the contrary that there is a vertex, say $z\in Z\sm \{z_1,z_2\}$. Then $C\cup (Z\sm \{z\})$ is  a clique cut-set of $G$ (by \cref{G2znbdXicard} and \ref{BZancom}), a contradiction.  This proves \ref{G2Zz1z2}. $\Diamond$

\begin{claim}\label{G2X2Y3str}
	 $G[X_2\cup Y_3]$ is $K_3$-free. Likewise,  $G[X_3\cup Y_2]$,  $G[X_1]$ and   $G[Y_1]$ are $K_3$-free.
\end{claim}
\no{\it Proof of \ref{G2X2Y3str}}.~Suppose not. Let $Q$ be a component of $G[X_2\cup Y_3]$ which contains a $K_3$ with vertices, say $p,q$ and $r$. We may assume that $V(Q)$ is anticomplete to $\{z_1\}$ (by \cref{G2znbdXicard}). Moreover,  we claim that $V(Q)$ is anticomplete to $X_1\cup X_3\cup Y_1\cup Y_2\cup Z\cup L$, and is given below.
\begin{enumerate}[leftmargin= 0.7cm, label={($\alph*$)}]\itemsep=0pt
	\item  First we will show that $V(Q)$ is a homogeneous set in $G[V(Q)\cup Y_1\cup Y_2]$. If not, then there are vertices, say $a,b\in V(Q)$ and $c\in Y_1\cup Y_2$ such that $ab,ac\in E(G)$ and $bc\notin E(G)$, and then  $\{b,a,c,v_3,z_1\}$ induces a $P_5$ (by \ref{BZancom}), a contradiction. So $V(Q)$ is a homogeneous set in $G[V(Q)\cup Y_1\cup Y_2]$.
	\item If there is a vertex, say $y\in Y_2$, which has a neighbor in $V(Q)$, then $\{y\}$ is complete to $\{p, q,r\}$ (by $(a)$), and then    $\{p,q,r,v_2,y\}$ induces a $K_5-e$; so  $V(Q)$ is anticomplete to $Y_2$.

	\item  $V(Q)$ is anticomplete to $Z$ (by $(b)$, \cref{G2xy2y3com} and by \ref{BZancom}), $V(Q)$ is anticomplete to $Y_1$ (by \cref{G2A1B2B3ancom}, $(a)$ and \ref{BiBi+1bigcomancom}), and $V(Q)$ is anticomplete to $X_1$ (by \cref{G2A1B2B3ancom} and \cref{G2A1A2ancom}).
	\item  By $(c)$, \ref{BZancom} and \cref{G2znbdXicard}, there is a vertex in $Z$, say $z'$ such that $\{z'\}$ is anticomplete to $V(Q)\cup X_3$. Then by using similar arguments in  $(a)$ and $(b)$, we see that $V(Q)$ is anticomplete to $X_3$.
	
	\item Finally we show that $V(Q)$ is anticomplete to $L$. Suppose not. Then there are adjacent vertices, say $t\in L$ and $u\in V(Q)$. Now since $\{t,u,v_2,v_3,y_2\}$ does not induce a $P_5$ (by $(b)$), we have $ty_2\in E(G)$. Then since $G[\{p,q,r,v_2,v_3,y_2,t\}]$ does not contain a $P_5$ (by $(b)$), $\{t\}$ is complete to $\{p,q,r\}$. But then $\{p,q,r,t,v_2\}$ induces a $K_5-e$, a contradiction. So $V(Q)$ is anticomplete to $L$.
\end{enumerate}	
	Thus we conclude that $V(Q)$ is anticomplete to $X_1\cup X_3\cup Y_1\cup Y_2\cup Z\cup L$, and hence $\{v_1,v_2\}$ is a clique cut-set separating $Q$ from the rest of the vertices  of $G$, a contradiction. This proves \ref{G2X2Y3str}. $\Diamond$

\begin{claim}\label{G2B1X3emp}
	If $L$ is not anticomplete to $X$, then either $X_1\cup X_2\cup Y_1=\es$ or $X_1\cup X_3\cup Y_1=\es$.
\end{claim}
\no{\it Proof of \ref{G2B1X3emp}}.~Since $L$ is not anticomplete to $X$, there are adjacent vertices, say $t\in L$ and $x\in X$. By \cref{G2LYcorel}, each vertex in $L$ is adjacent to one of $y_2$ and $y_3$, and $L$ is complete to $\{z_1,z_2\}$ (by \ref{G2LZcom}). So $x\notin X_1$ (by \ref{AiBi+1homset} and \cref{G2A1B2B3ancom}), and we may assume that $x\in X_2$. Then  $X_1=\es$ (by \ref{Atcorel}) and $xy_3\in E(G)$ (by \cref{G2X2Y3com}), and hence $Y_1=\es$ (by \cref{G2A1B2B3ancom}). Now it is enough to show that $X_3=\es$. Suppose not, and let $x'\in X_3$. We show that $X\cup Y\cup L$ is a clique.
  First we show that $X= \{x,x',y_2,y_3\}$.
Observe that $xx', x't\in E(G)$ (by \ref{Atcorel}), and then  $x'y_2\in E(G)$ (by \cref{G2X2Y3com}). So $X_2\cup X_3$ is complete to $\{t\}$, and $X_2$ is complete to $X_3$ (by \ref{Atcorel}). Thus $X_2$, $Y_3$, $\{t\}$  and $\{y_2\}$ are complete to each other (by \cref{G2X2Y3com}); so  $X_2\cup Y_3$ is a clique (by \ref{4setsclique}), and hence $X_2=\{x\}$ and $Y_3=\{y_3\}$ (by \ref{G2X2Y3str}).
			Likewise, $X_3=\{x'\}$ and $Y_2=\{y_2\}$.  So $X= \{x,x',y_2,y_3\}$.  Next we will show that $L$ is complete to $X$. By \ref{AiBi+1homset} and \ref{Atcorel}, it is enough to show that  $L$ is complete to $\{x,x'\}$. Suppose not, and let $t'\in L$ be such that $t'$ is not complete to $\{x,x'\}$. So by \ref{Atcorel}, $\{t'\}$ is anticomplete to $\{x,x'\}$. Recall that, by \cref{G2LYcorel}, $t'$  is adjacent to one of $y_2$ and $y_3$. If $t'y_2\in E(G)$, then we may assume that  $x'z_1\notin E(G)$ (by \ref{aZcard}), and then $\{x',y_2,t',z_1,v_2\}$ induces a $P_5$ (by \cref{G2LZcom}); so $t'y_2\notin E(G)$ and hence $t'y_3\in E(G)$. Then we assume that $xz_1\notin E(G)$ (by \ref{aZcard}), and then $\{x,y_3,t',z_1,v_2\}$ induces a $P_5$ (by \cref{G2LZcom}), a contradiction. So $L$ is complete to $\{x,x',y_2,y_3\}$.    So by  above arguments and by \ref{4setsclique}, we conclude that $X\cup Y\cup L$ is a clique. Thus $V(G)$ can be partitioned in two disjoint cliques, namely, $X\cup Y\cup L$ and $C\cup Z$, and hence $G$ is the complement of a bipartite graph, a contradiction. This proves \ref{G2B1X3emp}. $\Diamond$

\medskip

By Theorem~\ref{P5K3prop}:$(a)$ and \cref{G2X2Y3str}, for $i\in \{1,2,3\}$:  we pick a maximum stable set from each $5$-ring-component of $G[X_i]$ (if exists), and let $A_i$ be the union of these sets. So $G[X_i\sm A_i]$ is a bipartite graph. Next, we pick a maximum stable set from each big-component of $G[X_i\sm A_i]$ (if exists), and let $B_i$ be the union of these sets. Let $X_i'=X_i\sm (A_i\cup B_i)$.  Also,  let $Y_i'$ be a maximal stable set in  $G[Y_i]$. Then  $Y_i\sm Y_i'$ is a stable set (by \ref{BZancom} and \ref{G2X2Y3str}).
		Now we claim the following:

\begin{claim}\label{G2XiXjbigcomancom}
For $j,\ell \in \{2,3\}$ and $j\neq \ell$,	$A_j\cup B_j$ is anticomplete to $X_{\ell}\cup Z$.
\end{claim}
\no{\it Proof of \cref{G2XiXjbigcomancom}}.~We will show for $j=2$. Suppose to the contrary there are adjacent vertices, say $a\in A_2\cup B_2$ and $b\in X_3\cup Z$. Moreover there is a vertex, say $c\in X_2'\sm A_2$, such that $ac\in E(G)$. Now $ay_3,cy_3\notin E(G)$ (by \ref{G2X2Y3str} and \ref{AiBi+1homset}), and so $b\in X_3$, $az_1,az_2\notin E(G)$ (by \cref{G2xy2y3com}). Also we may assume that $bz_1\notin E(G)$ (by \cref{G2znbdXicard}). Then since $\{a,b,v_3,v_1,y_3\}$  does not induce a $P_5$, we have $by_3\in E(G)$, and then $\{a,b,y_3,v_1,z_1\}$ induces a $P_5$, contradiction. This proves \ref{G2XiXjbigcomancom}. $\Diamond$

\medskip
By \ref{G2LZcom}, $L$ can be partitioned in two stable sets, say $L_1$ and $L_2$, and we may assume that if $L$ is not anticomplete to $X$, then $X_1\cup X_3\cup Y_1=\es$ (by \ref{G2B1X3emp}).  Now we define the following sets:  $S_1:=X_2'\cup B_3\cup Y_1'\cup \{v_1\}$,
			$S_2:=B_1\cup X_3'\cup L_2\cup \{v_2\}$, $S_3:=X_1'\cup B_2\cup L_1\cup \{v_3\}$, $S_4:=A_1\cup A_3\cup Y_2'\cup (Y_1\sm Y_1')\cup (Y_3\sm Y_3')\cup \{z_1\}$ and $S_5:=A_2\cup (Y_2\sm Y_2')\cup Y_3'\cup \{z_2\}$. Then $V(G)=\cup_{j=1}^5S_j$ (by \ref{G2Zz1z2}), and we claim the following:
			
			\begin{claim}\label{G2-5stable}
				$S_1, S_2, \ldots, S_5$ are stable sets.
			\end{claim}
			\no{\it Proof of \ref{G2-5stable}}.~Clearly $S_1$ is a stable set (by \cref{G2A1B2B3ancom} and \ref{G2XiXjbigcomancom}), and $S_2$ is a stable set (by \cref{G2A1A2ancom}, \ref{G2X1L} and \ref{G2B1X3emp}). 			Now if there are adjacent vertices, say $a\in B_2$ and $t\in L_1$, then $ay_3\in E(G)$ (by \ref{G2LZcom} and \cref{G2X2Y3com}:$(i)$), and for any neighbor of $a$ in $X_2'$, say $x$, we have $xy_3\in E(G)$ (by \ref{AiBi+1homset}), and hence $\{x,a,y_3\}$ induces a $K_3$ in $G[X_2\cup Y_3]$, a contradiction to \ref{G2X2Y3str}; so $B_2\cup L_1$ is a stable set. This implies that $S_3$ is a stable set (by \cref{G2A1A2ancom} and \ref{G2X1L}).

			Next  $Y_2'\cup (Y_1\sm Y_1')\cup (Y_3\sm Y_3')\cup \{z_1\}$ is a stable set (by \ref{BZancom}, \ref{BiBi+1bigcomancom} and by the definition of $Y_i'$),  and $A_1\cup A_3$ is anticomplete to $(Y_1\sm Y_1')\cup (Y_3\sm Y_3')\cup \{z_1\}$ (by \cref{G2A1B2B3ancom},  \cref{G2X1Z}, \ref{K4XiYihomset} and \ref{G2XiXjbigcomancom}). Since every vertex of $A_3$ has a neighbor in $X_3'$,  $A_3$ is anticomplete to $Y_2'$ (by \ref{AiBi+1homset} and \ref{G2X2Y3str}), and so $A_1\cup A_3$ is anticomplete to $Y_2'$ (by \cref{G2A1B2B3ancom}). Also $A_1$ is anticomplete to $A_3$ (by \cref{G2A1A2ancom}). Thus  we conclude that $S_4$ is a stable set. Likewise, $S_5$ is also a stable set. This proves \ref{G2-5stable}. $\Diamond$

\smallskip
So we conclude that $\chi(G)\leq 5$ (by \ref{G2-5stable}). This completes the proof of \cref{cont-H2}.
\end{proof}

\subsubsection{$(P_5, K_5-e, F_1, F_2)$-free graphs that contain an $F_3$}\label{GcontF3}

Let $G$ be a connected $(P_5,K_5-e,F_1, F_2)$-free graph which has no clique cut-set. Suppose that $G$ contains an  $F_3$ with vertices and edges as shown in Figure~\ref{fig-F123}. Let $C := \{v_1,v_2,v_3\}$. Then, with respect to $C$, we define the sets $X,Y,Z$ and $L$ as in Section~\ref{genprop}, and we use the properties in Section~\ref{genprop}. Clearly  $y_2\in Y_2$, $y_3\in Y_3$  and $z_1,z_2\in Z$  so that $Y_2$, $Y_3$ and $Z$ are nonempty. Recall that $C\cup Z$ is a clique (by \ref{Zclq}), and that $Y$ is anticomplete to $Z$ (by \ref{BZancom}).  Further the graph $G$ has some more properties which we give in two lemmas below.

\begin{lemma}For $i\in \{1,2,3\}$,  the following hold:

\vspace{-0.2cm}
\begin{lemmalist}
\item  \label{H3YiYi+1ancom}
			$Y_i$ is anticomplete to $Y_{i+1}\cup Y_{i-1}$.
\item \label{H3A1B2B3ancom}
			$X_1$ is anticomplete to $Y_2\cup Y_3$. Likewise, if $Y_1\neq \es$, then for all  $i\in \{1,2,3\}$, $X_i$ is anticomplete to $Y_{i+1}\cup Y_{i-1}$.
\item \label{H3XiYcorel}
			Either $X_i$ is anticomplete to $Y_{i+1}\cup Y_{i-1}$ or $X_i$ is anticomplete to $Y_i$.

\item \label{H3lYzcorel}
			If a vertex of $L$ has a neighbor in $Y_i$, then it is complete to $Y\cup Z$.  Moreover, $G[Y_i]$ is a bipartite graph and $Z=\{z_1,z_2\}$.

\item \label{H3X2X3empLcorel}
			If one of $X_2$ and $X_3$ is empty, then $L$ is not anticomplete to $Y.$

\end{lemmalist}
\end{lemma}
\begin{proof} $(i)$:~If there are adjacent vertices, say $y\in Y_{i}$ and $y'\in Y_{i+1}\cup Y_{i-1}$,  then $\{v_1,v_2,v_3,y,y',z_1,z_2\}$ induces an $F_2$. So \cref{H3YiYi+1ancom} holds. $\diamond$

\smallskip
\no{$(ii)$}:~The proof is similar to \cref{G2A1B2B3ancom} by using \cref{H3YiYi+1ancom}, and   we omit the details. $\diamond$

\smallskip
\no{$(iii)$}:~By \cref{H3A1B2B3ancom}, it is enough to show for $i\in \{2,3\}$. We will show for $i=2$. Suppose to the contrary that the assertion is not true. Then, by using \cref{H3A1B2B3ancom}, we may assume that there are vertices, say $x,x'\in X_2$, $y\in Y_3$ and $y'\in Y_2$ such that $xy,x'y'\in E(G)$. Also we may assume that $xz_1\notin E(G)$ (by \ref{aZcard}). Then since $\{y,x,y',v_3,z_1\}$ does not induce a $P_5$, we have $xy'\notin E(G)$; so $x\neq x'$. Likewise, $x'y\notin E(G)$.  Then $xx'\notin E(G)$ (by \ref{AiBi+1homset}), and then $\{x,y,v_1,y',x'\}$ induces a $P_5$ (by \cref{H3YiYi+1ancom}), a contradiction.
		So \cref{H3XiYcorel} holds. $\diamond$

\smallskip
\no{$(iv)$}:~Clearly the first assertion follows from \Cref{H3YiYi+1ancom} and \ref{F2azancomLcorel}. Now if a vertex in $L$, say $t$,    has a neighbor in $Y_i$, then since $\{t,v_{i+1}\}$ is complete to $Y_i$, $G[Y_i]$ is a bipartite graph (by \ref{BZancom} and \ref{uvnonadjnbd}), and since $\{t\}$ is complete to $Z$, we have $Z=\{z_1,z_2\}$ (by \ref{lZcard}). This proves \cref{H3lYzcorel}. $\diamond$

\smallskip
\no{$(v)$}:~We may assume that $X_2=\es$. Suppose to the contrary that $Y$ is anticomplete to $L$. By \cref{H3XiYcorel}, $X_3$ is anticomplete to one of $Y_2$ or $Y_3$. Then by \cref{H3YiYi+1ancom} and \cref{H3A1B2B3ancom}, either $Y_2$ is anticomplete to $V(G)\sm (Y_2\cup \{v_1,v_3\})$ or $Y_3$ is anticomplete to $V(G)\sm (Y_3\cup \{v_1,v_2\})$. But  now  $C$ is a clique cut-set in $G$ separating one of $Y_2$ or $Y_3$ with the rest of the vertices in $G$, a contradiction. So $L$ is not anticomplete to $Y$. This proves \cref{H3X2X3empLcorel}.
\end{proof}

\begin{lemma}\label{H3-case1}
			If there is an $i\in \{1,2,3\}$ such that $X_i$ is not anticomplete to $Y_{i+1}\cup Y_{i-1}$, then $\chi(G)\leq 5$.
		\end{lemma}

\begin{proof}
		By \cref{H3A1B2B3ancom}, we may assume that $i=2$, and  there are adjacent vertices, say $p\in X_2$ and $q\in Y_1\cup Y_3$. Again, by \cref{H3A1B2B3ancom}, we  have $Y_1=\es$ and so $q\in Y_3$. Also, $X_2$ is anticomplete to $Y_2$ (by \cref{H3XiYcorel}).  Moreover, we  claim that:
\begin{equation} \label{H3X3emp}
			\mbox{$X_3=\es$.}
  \end{equation}

  \no{\it Proof of $(\ref{H3X3emp})$}:~Suppose not, and let $x\in X_3$. By \ref{aZcard}, we may assume that $xz_1\notin E(G)$. Since $\{v_1,v_2,v_3, z_1,z_2,$ $ p,q\}$ does not induce an $F_2$, $pz_1\notin E(G)$. Then   since $\{p,q,v_1,v_3,x\}$ or  $\{p,q,x,v_3,z_1\}$ does not induce a $P_5$, we have $px,qx\in E(G)$. Thus by \cref{H3XiYcorel},   $xy_2\notin E(G)$. But then $\{x,p,v_2,v_1,y_2\}$ induces a $P_5$, a contradiction. So (\ref{H3X3emp}) holds. $\Diamond$
	
\smallskip	
		By (\ref{H3X3emp}) and by \cref{H3X2X3empLcorel}, there is a vertex, say $t\in L$,   that   has a neighbor in $Y.$  Then by \cref{H3lYzcorel}, $\{t\}$ is complete to $Y,$ $G[Y_2]$ is a bipartite graph, and $Z=\{z_1,z_2\}$.  Then since for any $x\in X_2$, $\{t,y_2,v_3,v_2,x\}$ does not induce a $P_5$,   $\{t\}$ is complete to $X_2$. So $X_2$ is complete to $Y_3$ (by \ref{AiBi+1homset}). Now for any $x'\in X_2$ which has a neighbor in $Z$, say $z_1$,  since  $\{v_1,v_2,z_1,v_3,z_2, x',q\}$ does not induce  an $F_2$ (by \ref{BZancom} and \ref{aZcard}), we see that $X_2$ is anticomplete to $Z$. Thus if there are adjacent vertices, say   $u$ and $v$ in $X_2$, then $\{u,v,y_3,v_2,t\}$  induces a $K_5-e$; so $X_2$ is a stable set. Likewise, $Y_3$ is a stable set. Also  by \ref{AiBi+1homset} and \cref{H3A1B2B3ancom}, $\{t\}$ is anticomplete to $X_1$, and hence $X_1=\es$ (by \ref{Atcorel}). So  we conclude that $X\cup Y\cup Z = X_2\cup Y_2\cup Y_3\cup \{z_1,z_2\}$, and $\chi(G[X\cup Y\cup Z])\leq 2$. Next we claim that:
\begin{equation} \label{H3LYcom}
			\mbox{\it $L$ is complete to $Y.$}
		\end{equation}

\vspace{-0.1cm}
		\no{\it Proof of $(\ref{H3LYcom})$}.~Suppose to the contrary that  there is a component of $G[L]$, say $Q$,  such that $V(Q)$ is not complete to $Y.$ Thus by \ref{AiBi+1homset} and \cref{H3lYzcorel}, $V(Q)$ is anticomplete to $Y$. So by \ref{AiBi+1homset} and \ref{Atcorel}, $V(Q)$ is anticomplete to $X\cup Y$, and hence $Z$ is a clique cut-set  in $G$ separating $C$ and $V(Q)$, a contradiction.   So (\ref{H3LYcom}) holds. $\Diamond$
		
\smallskip
		By (\ref{H3LYcom}), $L$ is complete to $\{y_2,y_3\}$. So  by \ref{uvnonadjnbd} and Theorem~\ref{P5K3prop}:$(a)$, we conclude that $\chi(G[C\cup L])\leq 3$, and hence $\chi(G)\leq 5$. This proves \cref{H3-case1}.
\end{proof}

  Now we prove the main theorem of this section, and is given below.

\begin{theorem}\label{caseH3}
		Let $G$ be a connected $(P_5, K_5-e,F_1,F_2)$-free graph. If $G$ contains an  $F_3$, then either $G$ has a clique cut-set or $G$ is the complement of a bipartite graph or $\chi(G)= 5$.	
	\end{theorem}
	\begin{proof}
		Let $G$ be a connected $(P_5,K_5-e,F_1,F_2)$-free graph. Suppose that $G$ contains an induced $F_3$ with vertices and edges as shown in Figure~\ref{fig-F123}. Let $C := \{v_1,v_2,v_3\}$. Then, with respect to $C$, we define the sets $X,Y,Z$ and $L$ as in Section~\ref{genprop}, and we use the properties in Section~\ref{genprop}. Clearly  $y_2\in Y_2$ and $y_3\in Y_3$, and $z_1,z_2\in Z$, so that $Y_2$, $Y_3$ and $Z$ are nonempty. Recall that $C\cup Z$ is a clique (by \ref{Zclq}), and that $Y$ is anticomplete to $Z$ (by \ref{BZancom}). We may assume that $G$ has no clique cut-set, and that $G$ is not the complement of a bipartite graph. Now since $\chi(F_3)=5$, we have $\chi(G)\geq 5$, it is enough to show that $\chi(G)\leq 5$.    By  \cref{H3-case1}, we may assume that for each $i\in \{1,2,3\}$, $X_i$ is anticomplete to $Y_{i+1}\cup Y_{i-1}$. Next we claim the following:
		\begin{claim}\label{H3XLancom}
			$X$ is anticomplete to $L$.
		\end{claim}
	\no{\it Proof of \ref{H3XLancom}}.~Suppose to the contrary there are adjacent vertices, say $x\in X$ and $t\in L$. If $x\in X_2$, then by \ref{AiBi+1homset}, $ty_3\notin E(G)$, so  $\{t\}$ is anticomplete to $Y$ (by \cref{H3lYzcorel}), and then one of  $\{t,x,y_2,v_1,y_3\}$,  $\{t,x,v_2,v_3,y_2\}$ induces a $P_5$; so $x\notin X_2$. Likewise, $x\notin X_3$. Thus $x\in X_1$. So by \ref{Atcorel}, $X_2\cup X_3=\es$. Now let $L':=\{t\in L\mid t \mbox{ has a neighbor in } X_1\}$. Then $L'$ is anticomplete to $Y_2\cup Y_3 \cup (L\sm L')$  (by \ref{AiBi+1homset}), and for any $y'\in Y_1$, since one of $\{t,x,v_1,v_3,y'\}$ or $\{t,x,y',v_3,y_2\}$ does not induce a $P_5$, we have  $Y_1=\es$. But, then $C\cup Z$ is a clique cut-set separating $X_1\cup L'$ and the rest of the vertices in $G$, a contradiction. So \ref{H3XLancom} holds. $\Diamond$
		\begin{claim}\label{H3LK3free}
			$L$ is complete to $Y\cup Z$, and $G[L]$ is a bipartite graph.
		\end{claim}
		\no{\it Proof of \ref{H3LK3free}}.~Since  $C\cup Z$ is not a clique cut-set in $G$, each component of $G[L]$  is not anticomplete to $Y.$ So by \cref{H3lYzcorel} and \ref{AiBi+1homset},  $L$ is complete to $Y\cup Z$. In particular, $L$ is complete to $\{y_2,y_3,z_1,z_2\}$. Now the second assertion follows from \ref{uvadjnbd} and \ref{uvnonadjnbd}.	This proves \ref{H3LK3free}. $\Diamond$

\smallskip
To prove the our theorem,  we first suppose that there is an index $j\in \{1,2,3\}$ such  that $X_{j+1}\cup X_{j-1}=\es$.
 Then by \cref{H3X2X3empLcorel},  $L\neq \es$ and let $t\in L$. Thus by \cref{H3lYzcorel} and \ref{H3LK3free},  $G[Y_i]$ is a bipartite graph for each $i$, and $Z=\{z_1,z_2\}$. For $i\in \{1,2,3\}$, let $Y_i'$ be a maximal stable set in $G[Y_i]$. Next  if $Y_j =\es$, then since $X_j$ is anticomplete to $Y_{j+1}\cup Y_{j-1}\cup L$ (by \ref{H3XLancom}), $C\cup Z$ is a clique cut-set in $G$ separating $X_j$ and the rest of the vertices, a contradiction; so $Y_j \neq \es$.  Also for any $x\in X_j$ and $y\in Y_j$,  since $\{x,v_j,v_{j+1},y,t\}$ does not induce a $P_5$ (by \ref{H3LK3free}), $X_j$ is complete to $Y_j$. Then since $X_j$ is complete to $Y_j\cup \{v_j\}$, $G[X_j]$ is $K_3$-free (by \ref{uvnonadjnbd}).  Now by using \cref{P5K3prop}:$(a)$, we pick a maximum stable set from each $5$-ring-component of $G[X_j]$ (if exists), and let $S$ be the union of these sets; so $\chi(G[X_j\sm S])\leq 2$. Also $S\cup (Y_j\sm Y_j')\cup \{z_1\}$ is a stable set (by \cref{K4XiYihomset} and \cref{K4XZhomset}).
		Then from  \ref{H3XLancom} and \ref{H3LK3free}, we have $\chi(G[(X_j\sm S)\cup L\cup \{v_{j+1},v_{j-1}\}])= 2$, and from \ref{BZancom} and \cref{H3YiYi+1ancom}, we have $\chi(G[S\cup (Y\sm (Y_1'\cup Y_2'\cup Y_3'))\cup \{z_1\}])=1$ and $\chi(G[Y_1'\cup Y_2'\cup Y_3'\cup \{z_2\}])=1$.
		So we conclude that $\chi(G) =  \chi(G[C\cup X_j\cup Y\cup \{z_1,z_2\}\cup L])\leq 5$, and we are done. So we may assume that for each $j\in \{1,2,3\}$, either $X_{j+1}\neq \es$ or $X_{j-1}\neq \es$.		To proceed further, we let $M:=\{x\in X_1\mid N(x)\cap X_2\neq \es\}$. Then:
		\begin{claim}\label{H3subcaseX2X3emp}
		The following hold:	(i) $X_2$ and $X_3$ are non-empty.   (ii) $X_2$ is complete to $X_3$. (iii) $X_2$ is complete to $Y_2$. Likewise, $X_3$ is complete to $Y_3$.
		\end{claim}
		\no{\it Proof of \ref{H3subcaseX2X3emp}}.~$(i)$:~Suppose not, and let $X_3=\es$. Then by our assumption $X_1,X_2\neq \es$. By \cref{H3X2X3empLcorel}, $L\neq \es$ and let $t\in L$.  Since for any $x\in X_1$ and $x'\in X_2$, $\{x,x',v_2,y_3,t\}$ does not induce a $P_5$ (by \ref{H3LK3free}),  $X_1$ is anticomplete to $X_2$. Now by \ref{H3XLancom}, since $C\cup Z$ is not a clique cut-set in $G$,  $X_1$ is not anticomplete to $Y_1$, and $X_2$ is not anticomplete to $Y_2$. Let $p\in X_1$, $q\in X_2$, $r\in Y_1$ and $s\in Y_2$ be such that $pr,qs\in E(G)$. Then $\{p,r,v_3,s,q\}$ induces a $P_5$ (by \cref{H3YiYi+1ancom}). So  $(i)$ holds. $\diamond$

\smallskip
\no{$(ii)$}:~Suppose to the contrary there are nonadjacent vertices, say $p\in X_2$ and $q\in X_3$. Then for any $u\in Y_2$ and $v\in Y_3$,  since $\{p,u,v_1,v,q\}$ does not induce a $P_5$ (by \cref{H3YiYi+1ancom}),  we see that either $\{p\}$ is anticomplete to $Y_2$ or $\{q\}$ is anticomplete to $Y_3$. We may assume that $\{p\}$ is anticomplete to $Y_2$.  Then for any $x\in X_1$,  since $\{x,p,v_2,v_3,y_2\}$ does not induce a $P_5$ (by \cref{H3A1B2B3ancom}),   $\{p\}$ is anticomplete to $X_1$. Likewise, $\{p\}$ is anticomplete to $X_3$.  But now if  $Q$ is the component of $G[X_2]$  that contains $p$, then $C\cup Z$ is a clique cut-set  separating  $V(Q)$ and the rest of the vertices in $G$ (by \ref{AiBi+1homset} and \ref{H3XLancom}), a contradiction. So $(ii)$ holds. $\diamond$

\smallskip
\no{$(iii)$}:~If there are nonadjacent vertices, say $x\in X_2$ and $y\in Y_2$, then for any $x_3\in X_3$ (such a vertex exists, by $(i)$), $\{x_3,x,v_2,v_1,y\}$ induces  a $P_5$ (by $(ii)$); so  $X_2$ is complete to $Y_2$. So  $(iii)$ holds.
$\Diamond$

\begin{claim}\label{H3X1MX3com}
		The following hold:	(i) $X_1\sm M$ is complete to $X_3$. (ii) Each vertex in $Z$ is complete to exactly one of $X_2$ and $X_3$. (iii) Each vertex in $Z$ is complete to exactly one of $X_3$ and $X_1\sm M$.
		\end{claim}
		\no{\it Proof of \ref{H3X1MX3com}}.~$(i)$:~If there are nonadjacent vertices, say $x\in X_1\sm M$ and $x'\in X_3$, then for any $x_2\in X_2$, $\{x',x_2,v_2,v_1,x\}$ induces a $P_5$ (by \cref{H3subcaseX2X3emp}:$(i)$ and \cref{H3subcaseX2X3emp}:$(ii)$). So $(i)$ holds. $\diamond$

\smallskip
\no{$(ii)$}:~Let $z\in Z$ and let $z\neq z_1$. If $\{z\}$ is complete to both $X_2$ and $X_3$, then there are vertices, say $x_2\in X_2$ and $x_3\in X_3$ such that $zx_2,zx_3\in E(G)$, and then $\{z,v_2,v_3,v_1,z_1,x_2,x_3\}$ induces an $F_2$ (by \ref{H3subcaseX2X3emp}:$(ii)$); so $\{z\}$ is not complete to  $X_2$ or $X_3$.  Suppose that there is a  vertex, say   $x\in X_2$ such that $zx\notin E(G)$.  Now if there is a vertex, say $x'\in X_3$  such that $x'z\notin E(G)$, then   $\{x,x',y_3,v_1,z\}$ induces a $P_5$ (by \ref{H3subcaseX2X3emp}).  So $(ii)$ holds.  $\diamond$

\smallskip
\no{$(iii)$}:~This follows from a similar argument in $(ii)$ by using $(i)$ instead of \ref{H3subcaseX2X3emp}.   $\Diamond$

\smallskip
	 By \ref{H3subcaseX2X3emp}:$(i)$, we let $x_2\in X_2$ and $x_3\in X_3$.	From \ref{H3X1MX3com}:$(ii)$, we may assume that $\{z_1\}$ is complete to $X_2$. Then $\{z_2\}$ is anticomplete to $X_2$ (by \ref{aZcard}). So $\{z_2\}$ is complete to $X_3$ (by \ref{H3X1MX3com}:$(ii)$) and  hence anticomplete to $X_1\sm M$ (by \ref{H3X1MX3com}:$(iii)$). Next we claim that:

\begin{claim}\label{H3case1X1X2corel}
			  $M$ is complete to $\{z_2\} \cup X_2$,  and  is anticomplete to $\{z_1\}\cup X_3$.
		\end{claim}
		\no{\it Proof of \ref{H3case1X1X2corel}}.~Let $m\in M$, and  let $x\in X_2$ be a neighbour of $m$. Then since $\{m,x,y_2,v_3,z_2\}$ does not induce a $P_5$ (by \ref{H3subcaseX2X3emp}:($iii$)), $mz_2\in E(G)$; so   $M$ is complete to $\{z_2\}$. Thus $M$ is anticomplete to $\{z_1\}$ (by \ref{aZcard}).
  Now if there are adjacent vertices, say $p\in X_3$ and $q\in M$, then $\{z_2,v_1,v_3,z_1,v_2,p,q\}$ induces an $F_2$, a contradiction; so $M$ is anticomplete to $X_3$. Next  if there are nonadjacent vertices, say $u\in M$ and $v\in X_2$, then  $\{x_3,v,v_2,v_1,u\}$ induces $P_5$ (by \ref{H3subcaseX2X3emp}:(ii) and since $ux_3\notin E(G)$ by the previous argument). This proves \ref{H3case1X1X2corel}.   $\Diamond$

		\begin{claim}\label{H3zcard}
			Then following hold: (i) $Z=\{z_1,z_2\}$. (ii) $Y_1=\es$. (iii) If $X_1\neq \es$, then $L=\es$.
		\end{claim}
		\no{\it Proof of \ref{H3zcard}}.~$(i)$:~If there is a vertex, say $z\in Z\sm \{z_1,z_2\}$, then $\{z\}$ is complete to exactly one of $X_2$ and $X_3$ (by \ref{H3X1MX3com}:$(ii)$) which is  a contradiction to \ref{aZcard}. This proves $(i)$. $\diamond$

\smallskip		
\no{$(ii)$}:~Suppose not. Then since $C$ is not a clique cut-set separating $Y_1$ and the rest of the vertices, there are adjacent vertices, say $y\in Y_1$ and $p\in X_1\cup L$.  Then by \ref{H3XLancom}, \ref{H3subcaseX2X3emp} and \ref{H3case1X1X2corel}, we observe the following: If $p\in L$, then $\{x_2,x_3,v_3,y,p\}$ induces $P_5$, and if $p\in M$, then $\{y,p,x_2,x_3,y_3\}$ induces a $P_5$, and if $p\in X_1\sm M$, then $\{y,p,x_3,x_2,y_2\}$ induces a $P_5$,  which are contradictions. So $(ii)$ holds.   $\diamond$

\smallskip		
\no{$(iii)$}:~Otherwise, for any $t\in L$ and $p\in X_1$, from \ref{H3LK3free} and \ref{H3case1X1X2corel}, $\{p,x_2,v_2,y_3,t\}$ induces a $P_5$ (if $p\in M$), and from \ref{H3LK3free} and \ref{H3X1MX3com}:$(i)$,  $\{p,x_3,v_3,y_2,t\}$ induces  a $P_5$ (if $p\in X_1\sm M$). $\Diamond$

\begin{claim}\label{H3YK3free}
			  $G[Y_2]$, $G[Y_3]$, $G[X_2]$, $G[X_3]$, $G[M]$ and $G[X_1\sm M]$   are bipartite.
		\end{claim}
		\no{\it Proof of \ref{H3YK3free}}.~We show that, up to symmetry,  $G[Y_2]$,   $G[X_2]$ and    $G[M]$   are bipartite. Recall that  $Y_2$ is complete to $\{x_2,v_1,v_3\}$ (by  \ref{H3subcaseX2X3emp}:$(iii)$),  $X_2$ is complete to $\{x_3,v_2,z_1\}$  (by \ref{H3subcaseX2X3emp}:$(ii)$), and $M$ is complete to $\{x_2,v_1,z_2\}$ (by \ref{H3case1X1X2corel}).  Now the proof follows from \ref{uvadjnbd} and \ref{uvnonadjnbd}.  $\Diamond$

\smallskip
By above claims, since $G[M\cup X_3\cup Y_2\cup \{v_2,z_1\}]$ and  $G[(X_1\sm M)\cup X_2\cup Y_3\cup \{v_3,z_2\}]$ are bipartite, if $X_1\neq \es$, then $\chi(G)\leq 5$ (by \ref{H3zcard}). So we may assume that $X_1=\es$. For $j\in \{2,3\}$, let $Y_j'$, $X_j'$ and $L'$ respectively denote a maximal stable set of   $G[Y_j]$, $G[X_j]$ and $L$. Then we define the following sets: $S_1:=X_2'\cup Y_3'\cup \{z_2\}$, $S_2:=X_3'\cup Y_2'\cup \{v_2\}$, $S_3:=(X_2\sm X_2')\cup L'\cup \{v_3\}$, $S_4:=(X_3\sm X_3')\cup (L\sm L')\cup \{v_1\}$  and $S_5:=(Y_2\sm Y_2')\cup (Y_3\sm Y_3')\cup \{z_1\}$. Then by above arguments,  $S_i$'s are stable sets whose union is $V(G)$. So $\chi(G)\leq 5$. This completes the proof of \cref{caseH3}.
	\end{proof}

 \subsection{($P_5,K_5-e,F_1,F_2,F_3$)-free graphs with $\omega\geq 4$} \label{mainthm-f123free}

In this section, we prove \cref{main-thm} assuming that $G$ is $(F_1,F_2,F_3)$-free (see Figure~\ref{fig-F123}).

\subsubsection{($P_5,K_5-e, F_1, F_2, F_3$)-free graphs with $\omega\geq 5$}

\begin{theorem}\label{mainthm-w5}
Let $G$ be a connected ($P_5,K_5-e,F_1,F_2,F_3$)-free graph  with $\omega(G)\geq 5$.  Then either $G$ is   the complement of a bipartite graph or $G$ has a clique cut-set or   $\chi(G)\leq 6$.
\end{theorem}

\begin{proof}
	Let $G$ be a connected ($P_5,K_5-e, F_1, F_2,F_3$)-free graph. We may assume that $G$ has no clique cut-set, and that $G$ is not the complement of a bipartite graph. By \cref{HVNfreeper}, we assume that $G$ contains an $\mathbb{H}_2$, say $K$. Let $V(K):=\{v_1,v_2,v_3,z_1,z_2, y_1\}$ where $\{v_1,v_2,v_3,z_1,z_2\}$ induces a $K_5$, and  $N_{K}(y_1)=\{v_2,v_3\}$.  Let $C:=\{v_1,v_2,v_3\}$. Then  with respect to $C$, we define the sets $X$, $Y$, $Z$ and $L$ as in Section~\ref{genprop}, and we use the properties in Section~\ref{genprop}. Clearly  $y_1\in Y_1$ and $z_1,z_2\in Z$  so that $Y_1$ and $Z$ are nonempty. Recall that $C\cup Z$ is a clique (by \ref{Zclq}), and that $Y$ is anticomplete to $Z$ (by \ref{BZancom}).  To proceed further, we let $Z_1:=\{z\in Z\mid\{z\}\mbox{ is anticomplete to }X_1\}$.	We show that $\chi(G) \leq 6$. First we show that:
	\begin{claim}\label{K5Y2Y3emp}
		The following hold: (i) $Y_2\cup Y_3=\es$. (ii) $X_1\neq \es$.  (iii)  $X_2\cup X_3$ is anticomplete to $Z$.  (iv) $|Z\sm Z_1|\leq 1$.
	\end{claim}
	\no{\it Proof of \cref{K5Y2Y3emp}}.~$(i)$:~If there is a vertex, say $y\in Y_2\cup Y_3$, then  $\{v_1,v_2,v_3,z_1,z_2,y,y_1\}$ induces   an $F_2$ or an $F_3$ (by \ref{BZancom}). So $(i)$ holds. $\diamond$
	
\smallskip
\no{$(ii)$:}~If $X_1=\es$, then $Z\cup \{v_2,v_3\}$ is a clique cut-set separating $\{v_1\}$ and the rest of the vertices of $G$ (by \ref{Zclq} and  $(i)$), a contradiction. So $(ii)$ holds. $\diamond$
	
\smallskip
\no{$(iii)$:}~If there are adjacent vertices, say $x\in X_2\cup X_3$ and $z'\in Z$ (and we may assume that $z'\neq z_1$, by \ref{aZcard}), then $\{z',v_2,v_3,v_1,z_1,y_1,x\}$ induces  an $F_2$ or an $F_3$, a contradiction. So  $(iii)$ holds. $\diamond$
	
\smallskip
\no{$(iv)$:}~If there are vertices, say $z,z'\in Z$ and $x,x'\in X_1$  such that $zx,z'x'\in E(G)$, then $\{v_1,z,z',v_2,v_3,x,x'\}$ induces  an $F_2$ or an $F_3$ (by   \ref{aZcard}), a contradiction.  So  $(iv)$ holds.
  $\Diamond$
	
\medskip
	Since $|Z|\geq 2$, we may assume that $z_1\in Z_1$ (by \cref{K5Y2Y3emp}:$(iv)$). Next we claim  the following:
\begin{claim}\label{K5LZcorel}
	Each vertex in $Z_1$ has a neighbor in $L$, and so $L\neq \es$. Moreover, the vertex-set of each component of $L$ is not anticomplete to $X_1\cup Y_1$.
\end{claim}
\no{\it Proof of \cref{K5LZcorel}}.~For any $z\in Z_1$, since  $C\cup(Z\sm \{z\})$ is not a clique cut-set separating $\{z\}$ and the rest of the vertices of $G$ (by \ref{BZancom} and  \cref{K5Y2Y3emp}:$(iii)$), we see that each vertex in $Z_1$ has a neighbor in $L$. Since $z_1\in Z_1$, we have $L\neq \es$. This proves the first assertion. To prove the second assertion, we let $Q$ be an arbitrary component of $G[L]$. Since $C\cup Z$ is not a clique cut-set separating $V(Q)$ and the rest of the vertices of $G$, we see that $V(Q)$ is not anticomplete to $X\cup Y$. So from \ref{Atcorel}, \cref{K5Y2Y3emp}:$(i)$ and \cref{K5Y2Y3emp}:$(ii)$, it follows that $V(Q)$ is not anticomplete to $X_1\cup Y_1$. This proves \cref{K5LZcorel}. $\Diamond$

\begin{claim}\label{K5LZcom}
$Z_1$ is complete to $L$.
\end{claim}
\no{\it Proof of \cref{K5LZcom}}.~Suppose to the contrary that there is a vertex in $Z_1$, say $z$, such that $L\sm N(z)\neq \es$. We first show that $X\cup Y\cup L$ is a clique, and is given below.

\vspace{-0.2cm}
\begin{enumerate}[leftmargin = 0.67cm, label= {($\alph*$)}]\itemsep=0pt
\item
First we observe that since every vertex of $L\sm N(z)$ has a neighbor in $X_1\cup Y_1$ (by \cref{K5LZcorel}), $X_1$, $Y_1$ and $L\sm N(z)$ are complete to each other (by \ref{F2azancomLcorel}).
 \item {\it $L\cap N(z)$ is complete to $L\sm N(z)$}:  By \cref{K5LZcorel}, let $t'\in L\cap N(z)$ be arbitrary. Then by \ref{AiBi+1homset} and \cref{K5LZcorel}, let $u\in X_1\cup Y_1$ be such that $ut'\in E(G)$. Since for any $v\in L\sm N(z)$, one of $\{v_2,z,t',u,v\}$ or $\{v_1,z,t',u,v\}$ does not induce a $P_5$ (by (a)),   $\{t'\}$ is complete to $L\sm N(z)$. Thus $L\cap N(z)$ is complete to $L\sm N(z)$.

     \item From $(a)$, $(b)$ and from \ref{AiBi+1homset}, $X_1\cup Y_1$ is complete to $L$, where  $L\cap N(z) \neq\es$ and $L\sm N(z)\neq \es$.
     \item   From $(c)$ and  \ref{Atcorel}, it follows that $X_1, X_2$ and $X_3$ are complete to each other, and $X$ is complete to $L$, and hence  $X_2\cup X_3$ is complete to $Y_1$ (by \ref{AiBi+1homset}).
         \item From $(c)$ and $(d)$, and from \cref{K5Y2Y3emp}:$(i)$, we see that $X, Y$, $L\cap N(z)$ and $L\sm N(z)$  are complete to each other, and since  these sets are nonempty,  we conclude that $X\cup Y\cup L$ is a clique (by \ref{4setsclique}).
     \end{enumerate}
 So $V(G)$ can be partitioned into two  cliques, namely, $X\cup Y\cup L$ and $C\cup Z$. Thus $G$ is the complement of a bipartite graph, a contradiction.   So \cref{K5LZcom} holds. $\Diamond$

 \begin{claim}\label{K5LK3free}
	$G[L]$ is $K_3$-free.
\end{claim}
\no{\it Proof of \cref{K5LK3free}}.~Let $Q$ be a component of $G[L]$. By \cref{K5LZcorel},  there is a vertex in $X_1\cup Y_1$, say $a$, which has a neighbor in $V(Q)$. Then  $\{a\}$ is complete to $V(Q)$ (by \ref{AiBi+1homset}). Now since $\{a,z_1\}$ is complete to $V(Q)$ (by  \cref{K5LZcom}),    $Q$ is $K_3$-free (by \ref{uvnonadjnbd}). Since $Q$ is arbitrary, $G[L]$ is $K_3$-free. This proves \cref{K5LK3free}. $\Diamond$

\begin{claim}\label{K5Z1card}
	$|Z_1|\leq 2$.
\end{claim}
\no{\it Proof of \cref{K5Z1card}}.~The proof follows from \ref{lZcard}, \ref{K5LZcorel} and \cref{K5LZcom}. $\Diamond$

\begin{claim}\label{K5X2X3emp}
	$X_2\cup X_3=\es$.
\end{claim}
\no{\it Proof of \cref{K5X2X3emp}}.~Suppose not.  If $X_2\cup X_3$ is anticomplete to $X_1\cup Y_1$, then  $X_2\cup X_3$ is anticomplete to $L$ (by \ref{Atcorel} and \cref{K5Y2Y3emp}:$(ii)$), and then $C\cup Z$ is a clique cut-set separating
$X_2\cup X_3$ and the rest of the vertices (by \cref{K5Y2Y3emp}:$(i)$), a contradiction. So we may assume that  $X_2$ is not anticomplete to $X_1\cup Y_1$. To proceed further, we let $X_2':=\{x\in X_2\mid N(x)\cap (X_1\cup Y_1)\neq \es\}$.
We first show that $X\cup Y_1\cup L$ is a clique    using a sequence of arguments given below.

\vspace{-0.2cm}
\begin{enumerate}[leftmargin=0.67cm, label= {($\alph*$)}]\itemsep=0pt
\item
{\it $X_2'$, $X_1$ and $Y_1$ are complete to each other}:~Let $u\in X_2'$. For any $y\in Y_1\cap N(u)$ and for any $x\in X_1$, since $\{u,y,v_3,v_1,x\}$ and $\{u,y,x,v_1,z_1\}$ do not induce $P_5$'s (by \ref{BZancom} and \cref{K5Y2Y3emp}:$(iii)$), $(Y_1\cap N(u))\cup \{u\}$ is complete to $X_1$. Likewise, $(X_1\cap N(u))\cup \{u\}$ is complete to $Y_1$.   Since one of $(X_1\cap N(u))$ and  $(Y_1\cap N(u))$ is nonempty, it follows that $X_2'$ is complete to $X_1\cup Y_1$,  and $X_1$ is complete to $Y_1$  (by \ref{AiBi+1homset}).

\item {\it $X$ is complete to $L$ and $X_2=X_2'$}:~Suppose not. Then there is a vertex, say $t\in L$ such that $t$ has a nonneighbor in $X$. By \ref{Atcorel}, we may assume that $\{t\}$ is anticomplete to $X$. Then for any $x\in X_1$ and $x'\in X_2'$,  since $\{x,x',v_2,z_1,t\}$ does not induce a $P_5$ (by $(a)$, \cref{K5Y2Y3emp}:$(iii)$ and \cref{K5LZcom}), we see that $X$ is complete to $L$. This implies that $X_2'$ is complete to $X_1$ (by \ref{Atcorel}) and so $X_2'=X_2$.

\item Since $X_2$ is complete $L$ (by $(b)$),  we have $Y_1$ is complete to $L$ (by \ref{AiBi+1homset})  and hence again from \ref{AiBi+1homset}, $Y_1$ is complete to $X_3$. Also  $X_1$ is complete to $X_3$  (by $(b)$ and  \ref{Atcorel}).
\item By $(a), (b)$ and $(c)$, we conclude that $X$, $Y_1$ and $L$ are complete to each other. Then since $Y, L\neq \es$ and $|X|\geq 2$ (by \cref{K5Y2Y3emp}:$(ii)$ and \cref{K5LZcorel}), it follows from \ref{4setsclique} that $X\cup Y_1\cup L$ is a clique.
\end{enumerate}
  So by \cref{K5Y2Y3emp}:$(i)$, $V(G)$ can be partitioned into two  cliques, namely, $X\cup Y\cup L$ and $C\cup Z$. Thus $G$ is the complement of a bipartite graph, a contradiction. So \cref{K5X2X3emp} holds. $\Diamond$

\medskip
By  \cref{K5Y2Y3emp} and \cref{K5X2X3emp}, $V(G)= C\cup X_1\cup Y_1\cup Z\cup L$. Now we claim  the following:

\begin{claim}\label{K5X1K3free}
	The vertex-set of each component of $G[X_1]$ is a homogeneous set in $G[X_1\cup Y_1\cup L]$, and so $G[X_1]$ is $K_3$-free.
\end{claim}
\no{\it Proof of \cref{K5X1K3free}}.~If there are vertices, say $a,b\in X_1$ and $p\in Y_1\cup L$ such that $ab,ap\in E(G)$ and $bp\notin E(G)$, then $\{b,a,p,v_2,z_1\}$ induces a $P_5$ (by \ref{BZancom} and \cref{K5LZcom}); so the first assertion holds. To prove the second assertion, we let $Q$ be a component of $G[X_1]$. Since $C\cup Z$ is not a clique cut-set separating $V(Q)$ and the rest of the vertices, $V(Q)$ is not anticomplete to $L\cup Y_1$. Let $t\in L\cup Y_1$ be such that $\{t\}$ is not anticomplete to $V(Q)$. So $V(Q)$ is complete to $\{t,v_1\}$, by the first assertion. Then from \ref{uvnonadjnbd}, $V(Q)$ is $K_3$-free. Since $Q$ is arbitrary, $G[X_1]$ is $K_3$-free. This proves \cref{K5X1K3free}. $\Diamond$

\begin{claim}\label{K5YK3free}
	The vertex-set of each component of $G[Y_1]$ is a homogeneous set in $G$,  and so $\chi(G[Y_1])\leq 2$.
\end{claim}
\no{\it Proof of \cref{K5YK3free}}.~If there are vertices, say $a,b\in Y_1$ and $p\in V(G)\sm Y_1$ such that $ab,ap\in E(G)$ and $bp\notin E(G)$, then  $\{b,a,p,v_1,z_1\}$ induces a $P_5$ (by \ref{BZancom} and \cref{K5LZcom}); so the first assertion holds. To prove the second assertion, let  $Q$ be a component of $G[Y_1]$. Since $\{v_2,v_3\}$ is not a clique cut-set separating $V(Q)$  and the rest of the vertices,   $V(Q)$ is not anticomplete to $L\cup X_1$. Let $q\in L\cup X_1$ be such that $\{q\}$ is not anticomplete to $V(Q)$. Thus $V(Q)$ is complete to $\{q,v_2,v_3\}$, by the first assertion. Thus by \ref{uvadjnbd} and \ref{uvnonadjnbd}, it follows that $Q$ is ($P_3, K_3$)-free, and hence  $\chi(Q)\leq 2$. This proves \cref{K5YK3free}, since $Q$ is arbitrary. $\Diamond$

\medskip
Now by using \cref{P5K3prop}:$(a)$ and \cref{K5X1K3free}, we pick a maximum stable set from each  $5$-ring-component of $G[X_1]$ (if exists),  and let $A$ be the union of these sets. Again by using \cref{P5K3prop}:$(a)$ and \cref{K5LK3free}, we pick a maximum stable set from each $5$-ring-component of $G[L]$ (if exists),  and let $B$ be the union of these sets. Note that $A$ and $B$ are stable sets,      $\chi(G[X_1\sm A])\leq 2$ and  $\chi(G[L\sm B])\leq 2$ (by \cref{P5K3prop}:$(a)$). Let $Y_1'$ be a maximal stable set of $G[Y_1]$.   Next we claim the following:
\begin{claim}\label{K5col}
	$A\cup B\cup (Y_1\sm Y_1')\cup (Z_1\sm \{z_1\})$ is a stable set.
\end{claim}
\no{\it Proof of \cref{K5col}}.~Clearly $A\cup (Z_1\sm \{z_1\})$ and $(Y_1\sm Y_1')\cup (Z_1\sm \{z_1\})$ are stable sets (by \ref{BZancom} and \cref{K5Z1card}).
 Suppose to the contrary that there are adjacent vertices, say $p,q\in A\cup B\cup (Y_1\sm Y_1')\cup (Z_1\sm \{z_1\})$.  First suppose that $p\in A$. Then $q\notin Z_1\sm \{z_1\}$. Let $Q$ be the $5$-ring-component of $G[X_1]$ such that $p\in V(Q)$. Since $q\in B\cup (Y_1\sm Y_1')$, there is a vertex, say $r\in (L\sm B)\cup Y_1'$ such that $qr\in E(G)$. Then by \ref{AiBi+1homset}, \cref{K5X1K3free} and \cref{K5YK3free}, it follows that $\{q,r\}$ is complete to $V(Q)$. So by \ref{uvadjnbd}, $Q$ is $P_3$-free, a contradiction.   So we may assume that $p\in B$ and $q\in (Y_1\sm Y_1')\cup (Z_1\sm \{z_1\})$. Let $Q'$ be the $5$-ring-component of $G[L]$ such that $p\in V(Q')$. If $q\in Y_1\sm Y_1'$, then there is a vertex in $Y_1'$, say $y$, such that $qy\in E(G)$, and hence $V(Q')$ is complete to $\{q,y\}$ (by \cref{K5YK3free} and \ref{AiBi+1homset}); so from \ref{uvadjnbd}, $Q'$ is $P_3$-free, a contradiction. So   $q\in Z_1\sm \{z_1\}$. Since  $V(Q')$ is complete to $\{z_1,q\}$ (by \cref{K5LZcom}), it follows from \ref{Zclq} and \ref{uvadjnbd} that $Q'$ is $P_3$-free, a contradiction.  So \cref{K5col} holds. $\Diamond$

\medskip
 So from \cref{K5Y2Y3emp}:$(iv)$,  \cref{K5YK3free} and  \cref{K5col}, we conclude that $\chi(G)\leq\chi(G[A\cup B\cup (Y_1\sm Y_1')\cup (Z_1\sm \{z_1\})])+\chi(G[(X_1\sm A)\cup \{z_1,v_2\}])+\chi(G[(L\sm B)\cup \{v_1,v_3\}])+\chi(G[Y_1'\cup (Z\sm Z_1)])\leq 1+2+2+1 = 6$. This completes the proof of \cref{mainthm-w5}.
 \end{proof}

\subsubsection{($P_5, K_5-e, F_1$)-free graphs with $\omega =4$}
 We begin with the following. Let $G$ be a connected ($P_5$, $K_5-e$, $K_5$, $F_1$)-free graph  which has no clique cut-set. Suppose that $\omega(G)=4$. So $G$ contains a $K_4$, say $K$ with vertices $\{v_1,v_2,v_3,z^*\}$. We let $C := \{v_1,v_2,v_3\}$. Then,  with respect to $C$, we define the sets $X,Y,Z$ and $L$ as in Section~\ref{genprop}, and we use the properties in Section~\ref{genprop}. Clearly  $z^*\in Z$, and so $Z\neq \es$.  For $i\in \{1,2,3\}$, we let $W_i:=\{x\in X_i \mid xz^*\in E(G)\}$, and let $L_1:=\{t\in L\mid N(t)\cap X=\es\}$.
Recall that $C\cup Z$ is a clique (by \ref{Zclq}), and that $Y$ is anticomplete to (by \ref{BZancom}).   	Moreover, the graph $G$ has some more properties which we give in a few lemmas below.

\begin{lemma}\label{K4Xcompprop}
	For $i\in \{1,2,3\}$,    the following hold:

\vspace{-0.2cm}
	\begin{lemmalist}\itemsep=0pt
 \item \label{K4XiK3free} If $Y_{i+1}\cup Y_{i-1}\neq \es$, and if $V(Q)$ is the vertex-set of a component of $G[X_i]$ which is anticomplete to $L$, then $Q$ is $K_3$-free.
		\item\label{K4F4Wicol} The vertex-set of each big-component of $W_i$ is anticomplete to $Y_{i+1}\cup Y_{i-1}$.
		\item\label{K4C5ringcolor} If $Q$ is a $5$-ring component of $G[X_i]$ (if exists), then $V(Q)$ is anticomplete to $Y_{i+1}\cup Y_{i-1}$, to the vertex-set of each big-component of $G[X_{i+1}]$,  and to the vertex-set of each big-component of $G[X_{i-1}]$.

\item\label{K4X1Y2comY3emp} If $X_i\sm W_i$ is not anticomplete to  $Y_{i+1}$, then $W_i\cup Y_{i-1}=\es$. Likewise, if $X_i\sm W_i$ is not anticomplete to  $Y_{i-1}$, then $W_i\cup Y_{i+1}=\es$.
				\end{lemmalist}
\end{lemma}
\begin{proof} 	We will show for $i=1$.

\smallskip
	\no{$(i)$:}~If $V(Q)$ is complete to $\{z^*\}$, then since   $G$ is $K_5$-free, clearly $Q$ is $K_3$-free. So by \cref{K4XZhomset},  $\{z^*\}$ is anticomplete $V(Q)$. Then since the set $\{u\in V(G)\sm V(Q)\mid N(u)\cap V(Q)\neq \es\}$ is not a clique cut-set separating $V(Q)$ and the rest of the vertices, there are nonadjacent vertices, say $p,q\in V(G)\sm V(Q)$ such that both $p$ and $q$ has neighbors in $V(Q)$.  Then  $\{p,q\}$ is complete to $V(Q)$ (by \ref{AiBi+1homset} and \Cref{K4XiYihomset}). So $Q$ is $K_3$-free (by \ref{uvnonadjnbd}). This proves \cref{K4XiK3free}. $\diamond$
	
\smallskip
\no{$(ii)$:}~If there are  vertices, say $w,w'\in W_1$ and $y\in Y_{2}\cup Y_{3}$ such that $ww', wy\in E(G)$, then  $\{w,w',v_1,y,z^*\}$ induces a $K_5-e$ (by \ref{AiBi+1homset} and \cref{K4XZhomset}). So \cref{K4F4Wicol} holds. $\diamond$
	
	\smallskip
\no{$(iii)$:}~If there is a vertex, say $y\in Y_{2}\cup Y_{3}$  such that $y$ has a neighbor in $V(Q)$, then $\{y\}$ is complete to $V(Q)$ (by \ref{AiBi+1homset}), and then $V(Q)\cup \{v_1,y\}$ induces a $K_5-e$; so $V(Q)$ is anticomplete to $Y_{2}\cup Y_{3}$. A similar proof holds for other assertions  and   we omit the details.  So  \cref{K4C5ringcolor} holds. $\diamond$

\smallskip
\no{$(iv)$}:~Suppose not, and let $w\in W_1\cup Y_3$. If $w\in Y_3$, then $G$ contains either a $P_5$ or an $F_1$ (and the proof is similar to the proof of \cref{G2A1B2B3ancom}). So we assume that $w\in W_1$. Since $X_1\sm W_1$ is not anticomplete to $Y_2$, there are vertices, say $x\in X_1$ and  $y\in Y_2$, such that $xy\in E(G)$ and $z^*x\notin E(G)$. Then by \cref{K4XZhomset}, $wx\notin E(G)$. Now since $\{x,y,v_3,z^*,w\}$ does not induce a $P_5$, $wy\in E(G)$, and then $\{x,y,w,z^*,v_2\}$ induces a $P_5$. So \cref{K4X1Y2comY3emp} holds.
\end{proof}

\begin{lemma}\label{K4-YL}
The following hold:

\vspace{-0.2cm}
		\begin{lemmalist}
\item\label{K4-z} $Z = \{z^*\}$.
			\item\label{K4Yibipart}For $i\in \{1,2,3\}$,  $G[Y_i]$ is the union of $K_2$'s and $K_1$'s.
			 			\item\label{K4L1diamondfreecol} $L\sm L_1$ is anticomplete to $L_1$. Moreover, $\chi(G[L])\leq 3$.
\item\label{K4XLcomcol}
	If there is an $i\in \{1,2,3\}$ such that $X_i$, $X_{i+1}$ and $L\sm L_1$ are nonempty, then $\chi(G)\leq 7$.
			\end{lemmalist}
	\end{lemma}
	\begin{proof}$(i)$:~Since $G$ is $K_5$-free, this follows from \cref{Zclq}. $\diamond$

\smallskip
	\no{$(ii)$}:~Since $G[Y_i\cup \{v_{i+1},v_{i-1}\}]$ does not induce a $K_5$, $G[Y_i]$ is $K_3$-free. Also since $G[Y_i]$ is $P_3$-free (by \ref{BZancom}), we see that $G[Y_i]$ is the union of $K_2$'s and $K_1$'s. This proves \cref{K4Yibipart}.  $\diamond$

		\smallskip
	\no{$(iii)$}:~Clearly  $L\sm L_1$ is anticomplete to $L_1$ (by \ref{AiBi+1homset}). To prove the second assertion, consider a component of  $G[L]$, say $Q$. Then since $C\cup Z$ is not a clique cut-set separating $V(Q)$ and the rest of the vertices (by \ref{Zclq}), there is a vertex, say $p\in X\cup Y$ which has a neighbor in $V(Q)$. Then $\{p\}$ is complete to $V(Q)$ (by \ref{AiBi+1homset}), and then since $G$ is $(K_5,K_5-e)$-free, $Q$ is $(K_4,K_4-e)$-free. Hence $\chi(Q) \leq 3$, by \cref{P5diamond-col}. This proves \cref{K4L1diamondfreecol}, since $Q$ is arbitrary. $\diamond$

 \smallskip
	\no{$(iv)$}:~We may assume that $i=1$.  Clearly  $\chi(G[Y \cup Z]) \leq 3$  (by \ref{BiBi+1bigcomancom}, \cref{K4-z} and \cref{K4Yibipart}), and $\chi(G[L_1])\leq 3$ (by \cref{K4L1diamondfreecol}). Since $L\sm L_1\neq \es$,  it follows from \ref{Atcorel} that  $X_1,X_2$, $X_3$ and $L\sm L_1$ are complete to each other. Then since $G$ is ($K_5, K_5-e$)-free, it follows from   \ref{uvadjnbd} that  $G[X_1]$, $G[X_2\cup X_3]$ and $G[L\sm L_1]$ are ($P_3, K_3$)-free, and hence bipartite. Also since $G$ is $K_5$-free, at least two of  $X_1$, $X_2\cup X_3$ and $L\sm L_1$ are stable sets.   Then since $\{v_2\}$ is anticomplete to $X_1$, $\{v_1\}$ is anticomplete to $X_2\cup X_3$, and $\{v_3\}$ is anticomplete to $L\sm L_1$, we conclude that $\chi(G[C\cup X\cup (L\sm L_1)])\leq 4$. Clearly  $L_1$ is anticomplete to $(L\sm L_1) \cup X$ (by \ref{AiBi+1homset} and by the definition of $L_1$). So $\chi(G)\leq \chi(G[Y\cup Z])+\chi(G[C\cup X\cup L])=3+4=7$. This proves \cref{K4XLcomcol}.
					\end{proof}

\begin{lemma}\label{XiXi-1yi4col}
	For $i\in \{1,2,3\}$: Define $M=Y_i$ if $Y_{i+1}$ and $Y_{i-1}$ are nonempty, otherwise let $M=\es$. If $L$ is anticomplete to $X_{i+1}\cup X_{i-1}$, and if $G[X_{i+1}]$ and   $G[X_{i-1}]$ are $K_3$-free, then $\chi(G[X_{i+1}\cup X_{i-1}\cup M\cup L\cup \{v_{i}\}])\leq 4$.
	\end{lemma}
\begin{proof}
	Let $i=1$.
	For $k\in \{2,3\}$, we pick a maximum stable set from each $5$-ring-component of $G[X_k]$   (by using Theorem~\ref{P5K3prop}:$(a)$), and let $A_k$ be the union of these sets. So $G[X_2\sm A_2]$ and $G[X_3\sm A_3]$ are bipartite graphs. For $k\in \{2,3\}$, we pick a maximum stable set from each big-component of $G[X_k\sm A_k]$, and let $B_k$ be the union of these sets.
	By \cref{K4Yibipart}, we pick a maximum stable set from each big-component of $G[Y_1]$, and let $D$ be the union of these sets. So $X_2\sm (A_2\cup B_2)$, $X_3\sm (A_3\cup B_3)$ and $Y_1\sm D$ are stable sets.  Also  $A_3\cup B_2$ is stable set (by \cref{K4C5ringcolor}), $B_2\sm W_2$ is anticomplete to $Y_1$  (by \cref{K4X1Y2comY3emp}), and  $B_2\cap W_2$ is anticomplete to $Y_1$ (by \cref{K4F4Wicol}). So  $A_3\cup B_2\cup D\cup \{v_1\}$ is a stable set (by \cref{K4C5ringcolor}). Likewise, $A_2\cup B_3\cup (Y_1\sm D)$ is also a stable set. To proceed further, we let $L':=\{t\in L\mid N(t)\cap D\neq \es\}$, and we claim the following:
	\begin{equation} \label{DL'corel}
\longbox{\it $L'$ is complete to the vertex-set of each big-component of $G[Y_1]$, and is anticomplete to $L\sm L'$. Moreover, $G[L']$ is a bipartite graph.}
	\end{equation}
	\no{\it Proof of $(\ref{DL'corel})$}.~Let $Q$ be a big-component of $G[Y_1]$, and so $Q= K_2$ (by \cref{K4Yibipart}). Suppose to the contrary that there are nonadjacent vertices, say  $t\in L'$ and $y\in V(Q)$. Let $y'\in D$ be such that $y't\in E(G)$. Since $Y_2\neq \es$, we let $y_2\in Y_2$. Then since $\{t, y', v_2,v_1,y_2\}$ does not induce a $P_5$, we have $y_2t\in E(G)$ (by \ref{BiBi+1bigcomancom}), and then $\{t,y_2,v_1,v_2,y\}$ induces a $P_5$, a contradiction; so $L'$ is complete to $V(Q)$.
	Then   $L'$ is anticomplete to $L\sm L'$ (by \ref{AiBi+1homset}).
	Since $L'$ is complete to $D$, there are adjacent vertices, say $p,q\in Y_1$ such that $\{p,q\}$ is complete to $L'$. So  $G[L']$ is $P_3$-free (by \ref{uvadjnbd}), and since  $G$ is $K_5$-free,  $G[L']$ is $K_3$-free. Hence $G[L']$ is a bipartite graph. This proves (\ref{DL'corel}). $\Diamond$

\smallskip	
By (\ref{DL'corel}), there are stable sets, say $L_1'$ and $L_2'$, such that $L'=L_1'\cup L_2'$. By \cref{K4L1diamondfreecol}, there are three stable sets, say $R_1, R_2$ and $R_3$ such that $L\sm L'=R_1\cup R_2\cup R_3$. Now define $S_1:= (X_2\sm (A_2\cup B_2))\cup L_1'\cup R_1$, $S_2:= (X_3\sm (A_3\cup B_3))\cup L_2'\cup R_2$, $S_3:= A_3\cup B_2\cup D\cup R_3\cup \{v_1\}$ and $S_4:=A_2\cup B_3\cup (Y_1\sm D)$. Then by above arguments and by (\ref{DL'corel}), we see that $S_1,S_2,S_3$ and $S_4$ are stable sets. So $\chi(G[X_2\cup X_3\cup M\cup L\cup \{v_1\}])\leq 4$. This proves \cref{XiXi-1yi4col}.
\end{proof}

\begin{lemma} \label{YiL-4col}
	If there is an $i\in \{1,2,3\}$ such that $Y_{i+1}\cup Y_{i-1} \neq \es$, then $\chi(G[Y_i\cup L\cup \{v_{i}\}])\leq 4$.
	\end{lemma}
\begin{proof}
The proof is similar to the proof of \cref{XiXi-1yi4col}, and we omit the details.
\end{proof}

\begin{lemma}\label{K4-F4}
	If $G$ contains an $F_4$, then $\chi(G)\leq 7$.
\end{lemma}
\begin{proof}
	Suppose that $G$ contains an $F_4$ with vertices and edges as shown in Figure~\ref{fig-F123}. Let $C := \{v_1,v_2,v_3\}$. Then,  with respect to $C$, we define the sets $X,Y,Z$ and $L$ as in Section~\ref{genprop}, and we use the properties in    Section~\ref{genprop}. We also use \Cref{K4-YL,K4Xcompprop,XiXi-1yi4col,YiL-4col}. Clearly  $y_2\in Y_2$ and $y_3\in Y_3$ so that $Y_2$ and $Y_3$ are nonempty. For each $i$, since $G[N(v_i)]$ is ($K_4-e$)-free, by Theorem~\ref{P5diamond-col}, we have $\chi(G[N(v_i)])\leq 3$. Recall that $V(G)\sm N(v_i) = \{v_i\}\cup X_{i+1}\cup X_{i+2}\cup Y_i\cup L$. Now if $L=L_1$, then  $\chi(G)\leq 7$ (by \Cref{XiXi-1yi4col}) and we are done. So we may assume that $L\sm L_1\neq \es$. Also using \cref{K4XLcomcol}, we may assume that there is an index $k\in \{1,2,3\}$ such that $X_{k+1}$ and $X_{k-1}$ are empty. Then from \cref{YiL-4col}, $\chi(G[Y_k\cup L\cup \{v_{k}\}])\leq 4$, and so $\chi(G)\leq \chi(G[N(v_k)])+ \chi(G[V(G)\sm N(v_k)]) \leq 7$. This proves \cref{K4-F4}.
\end{proof}

\begin{lemma}\label{K4-F5}
	If $G$ is $F_4$-free, and contains an $F_5$, then  $\chi(G)\leq 7$.
\end{lemma}
\begin{proof}
	 Suppose that $G$ contains an   $F_5$ with vertices and edges as shown in Figure~\ref{fig-F123}. Let $C := \{v_1,v_2,v_3\}$. Then, with respect to $C$, we define the sets $X,Y,Z$ and $L$ as in Section~\ref{genprop},  and we use the properties in    Section~\ref{genprop}. We also use \Cref{K4-YL,K4Xcompprop,XiXi-1yi4col,YiL-4col}. Clearly  $x_1\in X_1$ and $y_2\in Y_2$ so that $X_1$ and $Y_2$ are nonempty. For each $i$, since $G[N(v_i)]$ is ($K_4-e$)-free, by Theorem~\ref{P5diamond-col}, we have $\chi(G[N(v_i)])\leq 3$.
	Now for any $y\in Y_1\cup Y_3$, since  $\{v_1,v_2,v_3,y_2,y,z^*\}$ does not induce an $F_4$,  $Y_1\cup Y_3=\es$.
	Also for any $x\in X_1\cup X_3$, since  $\{v_1,v_2,v_3,x,y_2,z^*\}$ does not induce an $F_4$,  $X_1\cup X_3$ is anticomplete to $\{z^*\}$.
	Since $\{v_1,v_3,z^*\}$ is not a clique cut-set separating $\{v_2\}$ and the rest of the vertices, we have $X_2\neq \es$. Now if $L\sm L_1\neq \es$, then  $\chi(G)\leq 7$ (by \Cref{K4XLcomcol}) and we are done. So we may assume that $L=L_1$. Then
	$G[X_1]$ and $G[X_3]$ are $K_3$-free (by \cref{K4XiK3free}). Next we claim the following:
	\begin{equation} \label{F5X2K3free} \longbox{\it If $X_1$ is not anticomplete to $Y_2$, then   $G[X_2]$ is $K_3$-free. Likewise, if $X_3$ is not anticomplete to $Y_2$, then $G[X_2]$ is $K_3$-free.}
	\end{equation}
	\no{\it Proof of $(\ref{F5X2K3free})$}.~Suppose to the contrary that there is a component, say $Q$,  that contains a $K_3$ induced by the vertices, say $\{p_1,p_2,p_3\}$. Since $G$ is $K_5$-free, we may assume that $p_1z^*\notin E(G)$. By our assumption, there are vertices, say $x\in X_1$ and $y\in Y_2$. Then since one of  $\{x,y,v_3,v_2,p_1\}$ or $\{x,y,p_1,v_2,z^*\}$ does not induce a $P_5$, we have $p_1x\in E(G)$. Then  $\{p_1,p_2,p_3,x,v_2\}$ induces a $K_5-e$ (by \ref{AiBi+1homset}), a contradiction. So (\ref{F5X2K3free}) holds. $\Diamond$
	
\smallskip
	 Now if $X_1\cup X_3$ is not anticomplete to $Y_2$, then from (\ref{F5X2K3free}), $G[X_2]$ is $K_3$-free, and so from \cref{XiXi-1yi4col}, we have $\chi(G[X_1\cup X_2\cup L\cup \{v_3\}])\leq 4$, and hence $\chi(G)\leq \chi(G[N(v_3)])+ \chi(G[X_1\cup X_2\cup L\cup \{v_3\}])\leq 7$. So we may assume that $X_1\cup X_3$ is anticomplete to $Y_2$.  Then since $\chi(G[Y_2\cup \{v_2\}])\leq 2$ (by \cref{K4Yibipart}), from \cref{XiXi-1yi4col}, it follows that $\chi(G[X_1\cup Y_2\cup X_3\cup \{v_2,z^*\}])\leq 4$. Since $N(v_2)\sm \{z^*\}=X_2\cup \{v_1,v_3\}$ (by \cref{K4-z}), we see that $N(v_2)\sm \{z^*\}$ is anticomplete to $L$. So from \cref{K4L1diamondfreecol}, it follows that $\chi(G[(N(v_2)\sm \{z^*\})\cup L])\leq 3$. Hence $\chi(G)\leq 7$. This proves \cref{K4-F5}.
\end{proof}

\begin{lemma}\label{K4-HVN}
If $G$ contains an HVN, then $\chi(G)\leq 7$.
\end{lemma}
\begin{proof}
We may assume that $G$ contains an HVN, say $K$, with vertex-set  $\{v_1,v_2,v_3,z^*, y_1\}$ such that $\{v_1,v_2,v_3,z^*\}$ induces a $K_4$ and $N_K(y_1)=\{v_2,v_3\}$.   Let $C := \{v_1,v_2,v_3\}$. Then, with respect to $C$, we define the sets $X,Y,Z$ and $L$ as in Section~\ref{genprop}, and we use the properties in    Section~\ref{genprop}. We also use \Cref{K4-YL,K4Xcompprop,XiXi-1yi4col,YiL-4col}.    Clearly   $y_1\in Y_1$ so that $Y_1$ is nonempty.  We may assume that, from \cref{K4-F4} and \cref{K4-F5}, $G$ is ($F_4, F_5$)-free.  Now for any $y\in Y_2\cup Y_3$, since $\{v_1,v_2,v_3,z^*,y_1,y\}$ does not induce an $F_4$, we have $Y_2\cup Y_3=\es$.
	Also for any $x\in X_2\cup X_3$, since  $\{v_1,v_2,v_3,z^*,y_1,x\}$ does not induce an $F_4$ or an $F_5$, we have $X_2\cup X_3=\es$.
	Further since $\{v_2,v_3,z^*\}$ is not a clique cut-set separating $\{v_1\}$ and the rest of the vertices, we have $X_1\neq \es$.
	Next we claim that:
	\begin{equation} \label{H1casecol}
		  \mbox{$\chi(G[X_1\cup Y_1\cup \{z^*\}])\leq 4$.}
	\end{equation}
	\no{\it Proof of $(\ref{H1casecol})$}.~First suppose that there is a vertex, say $x\in X_1$ such that $xz^*\in E(G)$. Then  for any $x'\in X_1\sm \{x\}$, since $\{v_1,v_2,v_3,x,z^*,x'\}$ does not induce an $F_5$, we see that $X_1$ is complete to $\{z^*\}$. Since $X_1$ is complete to $\{v_1,z^*\}$ and since $G$ is  ($K_5, K_5-e$)-free, $G[X_1]$ is $(P_3,K_3)$-free; so $\chi(G[X_1])\leq 2$. Then from \cref{K4Yibipart}, $\chi(G[X_1\cup Y_1\cup \{z^*\}])\leq 4$ (by \ref{BZancom}), and we are done. So we may assume that $X_1$ is anticomplete to $\{z^*\}$. By \cref{K4Yibipart}, we pick a maximum stable set from each big-component of $G[Y_2]$, and let $D$ be the union of these sets. To proceed further, we let $X_1':=\{x\in X_1| N(x)\cap D\neq \es\}$, and consider a component of $G[X_1']$, say $Q$. By \cref{K4XiYihomset}, there are adjacent vertices, say  $a,b\in Y_1$  such that $\{a,b\}$ is complete to $V(Q)$. Then since $G$ is  ($K_5, K_5-e$)-free, $Q$ is $(P_3,K_3)$-free, and so $\chi(Q)\leq 2$. Hence $\chi(G[X_1'\cup D])\leq 3$. Since $G[X_1\sm X_1']$ is complete to $\{v_1\}$, $G[X_1\sm X_1']$ is $(K_4-e)$-free and hence $\chi(G[X_1\sm X_1'])\leq 3$ (by \cref{P5diamond-col}). Since $Y_1\sm D$ is a stable set and since $X_1\sm X_1'$ is anticomplete to $X_1'\cup D$ (by \cref{K4XiYihomset}), we conclude that
	$\chi(G[X_1\cup Y_1\cup \{z^*\}])\leq \chi(G[(X_1'\cup D)\cup (X_1\sm X_1')])+\chi(G[(Y_1\sm D)\cup \{z^*\}]) \leq 3+1= 4$ (by \ref{BZancom}). This proves (\ref{H1casecol}). $\Diamond$
	
\smallskip
	 From (\ref{H1casecol}), and  from \cref{K4-z} and \cref{K4L1diamondfreecol}, we have  $\chi(G)\leq \chi(G[C\cup L])+\chi(G[X_1\cup Y_1\cup \{z^*\}])\leq 3+4=7$. This proves \cref{K4-HVN}.
\end{proof}

Now we prove the main theorem of this section, and is given below.

 \begin{theorem}\label{mainthm-w4}
Let $G$ be a ($P_5,K_5-e,F_1$)-free graphs with $\omega(G)= 4$.  Then either $G$ is   the complement of a bipartite graph or $G$ has a clique cut-set or   $\chi(G)\leq 7$.
\end{theorem}
\begin{proof}Let $G$ be a connected $(P_5,K_5-e, F_1)$-free graph with $\omega(G)=4$.  We may assume that $G$ has no clique cut-set and that $G$ is not the complement of a bipartite graph. By \cref{HVNfreeper}, we may assume that $G$ contains an HVN. Now the theorem follows from \cref{K4-HVN}. \end{proof}

\section{Proof of \cref{main-thm}}\label{mainthm-proof}
Let $G$ be a connected $(P_5, K_5-e)$-free graph with $\omega(G)\geq 4$. If $G$ contains one of   $F_1$, $F_2$ or $F_3$, then the theorem follows from \cref{cont-HG,cont-H2,caseH3}. So we may assume that
$G$ is ($F_1, F_2, F_3$)-free. Now if $\omega(G)\geq 5$, then the theorem follows from \cref{mainthm-w5}, and if $\omega(G) =4$, then the theorem follows from \cref{mainthm-w4}. This completes the proof of \cref{main-thm}. \hfill$\Box$

\end{document}